\documentclass[12pt,a4paper]{amsart}
\usepackage[utf8]{inputenc}
\usepackage{amsmath,amssymb,amsfonts,amsthm,url,graphicx}
\usepackage{cite}
\usepackage{multirow}

\usepackage{caption,tabu}

\newcommand{\llambda}{\log\lambda}

\newtheorem{theorem}{Theorem}
\newtheorem{proposition}[theorem]{Proposition}
\newtheorem{lemma}[theorem]{Lemma}
\theoremstyle{definition}
\newtheorem{remark}[theorem]{Remark}
\newtheorem{conjecture}{Conjecture}
\newtheorem*{ackno}{Acknowledgement}
\numberwithin{equation}{section}
\numberwithin{theorem}{section}

\author[A. S. Feigenbaum]{Ahram S. Feigenbaum\textsuperscript{\dag}}
\address[AF\&DH]{Department of Mathematics, Vanderbilt University,
  1326 Stevenson Center, Nashville, TN 37240, USA}
\email{ahram.s.feigenbaum@Vanderbilt.Edu}
\email{doug.hardin@Vanderbilt.Edu}
\author[P. J. Grabner]{Peter J. Grabner\textsuperscript{*}}
\address[PG]{Institute of Analysis and Number Theory,
  Graz University of Technology,
  Kopernikusgasse 24.
8010 Graz,
Austria}
\email{peter.grabner@tugraz.at}
\thanks{\textsuperscript{*}This author is supported by the Austrian Science Fund
  FWF   project F5503 part of the Special Research Program
  (SFB) ``Quasi-Monte Carlo Methods: Theory and Applications''}
 \thanks{\textsuperscript{\dag}The research of these authors was supported, in part,
by the U. S. National Science Foundation under grant DMS-1516400.}

\author[D. P. Hardin]{Douglas P. Hardin\textsuperscript{\dag}}

\begin{document}
\title[Fourier eigenfunctions]
{Eigenfunctions of the Fourier Transform with specified zeros}
\date{\today}
\maketitle

\section{Introduction}\label{sec:intro}
The sphere packing problem is one of those mathematical problems which are
easy to state and notoriously difficult to solve. A sphere packing
$\mathcal{P}$ of $\mathbb{R}^d$ is a collection of congruent non-overlapping
balls. Its (upper) density
\begin{equation*}
  \Delta(\mathcal{P})=\limsup_{r\to\infty}\frac{\mathrm{vol}_d
    \left(\mathcal{P}\cap
      B(\mathbf{0},r)\right)}{\mathrm{vol}_d\left(B(\mathbf{0},r)\right)}
\end{equation*}
is the amount of space covered by $\mathcal{P}$ (here $B(\mathbf{0},r)$ denotes
the euclidean ball of radius $r$ centered at $\mathbf{0}$; $\mathrm{vol}_d$ is
the $d$-dimensional Lebesgue measure). The sphere packing problem asks for the
maximal value of $\Delta(\mathcal{P})$ and for which packing it is
attained. Until 2017 the answer was only known for dimensions $1$, $2$
(see~\cite{Fejes1943:uber_dichteste_kugellagerung}), and $3$. In 1611 Kepler
conjectured that no sphere packing in $\mathbb{R}^3$ has density greater than
$\pi\sqrt{2}/6$ which is the density of the \emph{face-centered cubic} lattice
and the \emph{hexagonal close packing}. The proof of the Kepler conjecture by
Hales \cite{Hales2005:proof_kepler_conjecture} was a major achievement and an
endpoint of a long development (for a historical exposition
see~\cite{Aste_Weaire2008:pursuit_perfect_packing}).

The solution of the sphere packing problem in dimension $8$ in March 2016 by
Viazovska \cite{Viazovska2017:sphere_packing_8} and soon after in dimension
$24$ by Cohn, Kumar, Miller, Radchenko, and Viazovska
\cite{Cohn-Kumar-Miller+2017:sphere_packing_24} brought an enormous
breakthrough in the application of linear programming techniques.  For a
comprehensive overview of the proof and more background information we refer to
\cite{Cohn2017:conceptual_breakthrough_packing,
  Laat_Vallentin2016:breakthrough_sphere_packing}.

The proof was based on earlier work by Cohn and Elkies
\cite{Cohn_Elkies2003:new_upper_bounds,Cohn2002:new_upper_bounds}, who provided
the underlying linear programming technique. More precisely, the following
result was proved there.
\begin{theorem}[Theorem~3.2 in \cite{Cohn_Elkies2003:new_upper_bounds}]
  Suppose $f : \mathbb{R}^d \to\mathbb{R}$ is an admissible function satisfying
  the following three conditions for some $r>0$:
  \begin{enumerate}
  \item $f(\mathbf{0})=\widehat{f}(\mathbf{0})>0$,
  \item $f(\mathbf{x})\leq0$ for $\|\mathbf{x}\|\geq r$
  \item $\widehat{f}(\mathbf{t})\geq0$ for all $\mathbf{t}\in\mathbb{R}^d$.
  \end{enumerate}
  Then the density of sphere packings in $\mathbb{R}^d$ is bounded above by
  $(r/2)^dB_d$, where $B_d$ denotes the volume of the $d$-dimensional unit
  ball.
\end{theorem}
A function $f$ is admissible, if there exists a constant $\delta>0$ such that
$|f(\mathbf{x})|$ and $|\widehat{f}(\mathbf{x})|$ are bounded above by a
constant times $(1+\|\mathbf{x}\|)^{-d-\delta}$. The admissibility condition
implies the validity the Poisson summation formula for $f$ which plays a
central role in the proof of the above result and further shows that the bound
is attained for the lattice packing
\begin{equation*}
  \mathcal{P}=\bigcup_{\mathbf{x}\in\Lambda}B\left(\mathbf{x},\frac r2\right)
\end{equation*}
for a lattice $\Lambda$, if and only if $f(\mathbf{x})=0$ for all
$\mathbf{x}\in\Lambda\setminus\{\mathbf{0}\}$ and $\widehat{f}(\mathbf{t})=0$
for all $\mathbf{t}\in\Lambda^*\setminus\{\mathbf{0}\}$, where $\Lambda^*$
denotes the dual lattice of $\Lambda$, and $r$ is the minimal distance of
$\Lambda$.

Of course, Schwartz functions
are admissible. It is an important feature of the space of real valued radial
Schwartz functions that every element $f$ can be written as
$f=f_++f_-$, where $\widehat{f}_+=f_+$ and $\widehat{f}_-=-f_-$ are
eigenfunctions of the Fourier transform. This is one of the key ingredients in
the construction of functions $f$ satisfying the assumptions of the theorem.

The functions constructed in
\cite{Cohn-Kumar-Miller+2017:sphere_packing_24,Viazovska2017:sphere_packing_8}
were the first examples of Fourier eigenfunctions with prescribed double zeros
at the distances of a lattice. The construction was based on Laplace transforms
of certain weakly holomorphic modular forms and quasimodular forms. The aim of
the present paper is to provide a unifying view on the modular forms behind
these constructions, to construct the Fourier eigenfunctions for all dimensions
divisible by $4$, and to show that the underlying modular and quasimodular
forms are uniquely determined by the requirements that their transform should
be an eigenfunction of the Fourier transform belonging to the Schwartz class.

Bourgain, Clozel, and Kahane
\cite{Bourgain_Clozel_Kahane2010:principe_dheisenberg} studied an uncertainty
principle for the last sign change of even functions and their Fourier
transforms on $\mathbb{R}$. 
More precisely, for dimension $d\geq1$ let $\mathcal{A}_+(d)$ denote the set of
functions $f:\mathbb{R}^d\to\mathbb{R}$ satisfying
\begin{enumerate}
\item $f,\widehat{f}\in L^1(\mathbb{R}^d)$ and $\widehat{f}$
      real valued, thus $f$ is even
\item $f$ is eventually non-negative, while $\widehat{f}(\mathbf{0})\leq0$
\item $\widehat{f}$ is eventually non-negative, while $f(\mathbf{0})\leq0$.
\end{enumerate}
Similarly, denote by $\mathcal{A}_-(d)$ denote the set of functions
$f:\mathbb{R}^d\to\mathbb{R}$ satisfying
\begin{enumerate}
\item $f,\widehat{f}\in L^1(\mathbb{R}^d)$ and $\widehat{f}$
      real valued, thus $f$ is even
\item $f$ is eventually non-negative, while $\widehat{f}(\mathbf{0})\leq0$
\item $\widehat{f}$ is eventually non-positive, while $f(\mathbf{0})\geq0$.
\end{enumerate}
For such functions define
\begin{equation*}
  r(f)=\inf\left\{R\geq0\mid f(\mathbf{x})\text{ has the same sign for }
  \|\mathbf{x}\|\geq R\right\}
\end{equation*}
and set
\begin{equation}
  \label{eq:A+A-}
  \begin{split}
    A_+(d)&=\inf_{f\in\mathcal{A}_+(d)}\sqrt{r(f)r(\widehat{f})}\\
    A_-(d)&=\inf_{f\in\mathcal{A}_-(d)}\sqrt{r(f)r(\widehat{f})}.
  \end{split}
\end{equation}
This question was originally motivated by the study of zeta functions of number
fields which have a real zero between $0$ and $1$.

Gon\c{c}alves, Oliveira e Silva, and Steinerberger
\cite{Goncalves_Oliveira_Steinerberger2017:hermite_polynomials_linear} studied
the problem further and proved that the extremal functions for the above
properties are eigenfunctions of the Fourier transform. Cohn and Gon\c{c}alves
\cite{Cohn_Goncalves2019:optimal_uncertainty_principle} used a construction
similar to the ones in
\cite{Cohn-Kumar-Miller+2017:sphere_packing_24,Viazovska2017:sphere_packing_8}
to provide the optimal function for the above uncertainty principle in
dimension $12$. They found $A_+(12)=\sqrt2$. Exact values of $A_+(d)$ and
$A_-(d)$ are known only for very few dimensions
(see~\cite{Cohn_Goncalves2019:optimal_uncertainty_principle}).

A point configuration $\mathcal{C}\subset\mathbb{R}^d$ is said to have density
$\rho$, if
\begin{equation*}
  \rho=\lim_{r\to\infty}\frac{\#(\mathcal{C}\cap B(\mathbf{0},r))}
  {\mathrm{vol}_d(B(\mathbf{0},r))},
\end{equation*}
meaning that $\mathcal{C}$ contains $\rho$ points per unit volume. For a
completely monotone function $p:(0,\infty)\to\mathbb{R}$ the $p$-energy of
$\mathcal{C}$ is given by
\begin{equation*}
  E_p(\mathcal{C})=\liminf_{r\to\infty}
  \frac1{\#\left(\mathcal{C}\cap B(\mathbf{0},r)\right)}
  \sum\limits_{\substack{\mathbf{x},\mathbf{y}\in\mathcal{C}\cap
     B(\mathbf{0},r)\\\mathbf{x}\neq\mathbf{y}}}p(\|\mathbf{x}-\mathbf{y}\|),
\end{equation*}
which can be viewed as a thermodynamic limit of the sum of all mutual
$p$-interactions of distinct points in $\mathcal{C}$. A configuration
$\mathcal{C}$ of density $\rho$ is called \emph{universally optimal}, if
it minimises $E_p(\cdot)$ amongst all configurations of density $\rho$ and for
all completely monotone functions $p$ simultaneously. Such configurations seem
to exist only for special values of the dimension; only very few examples are
known.

Radchenko and Viazovska \cite{Radchenko_Viazovska2019:fourier_interpolation}
proved a remarkable interpolation theorem for functions on the real line, with
prescribed values of $f$ and $\widehat f$ in the points $\pm\sqrt{n}$
($n\in\mathbb{N}_0$). This idea was taken further by Cohn, Kumar, Miller, Radchenko,
and Viazovska \cite{Cohn_Kumar_Miller+2019:universal_optimality} in their proof
of universal optimality of the $E_8$ and Leech lattices in respective
dimensions $8$ and $24$.  The main ingredient of their proof is an
interpolation formula for radial Schwartz functions in these dimensions, which
allows to interpolate values and first derivatives of $f$ and $\hat f$ in the
points $\sqrt{2n}$ ($n\in\mathbb{N}$).

As we were completing this manuscript we became aware of the work by Rolen and
Wagner~\cite{Rolen_Wagner2020:note_schwartz_functions}, who studied similar
questions for dimensions divisible by $8$. They were focused on applications
for proving packing bounds in these dimensions. These bounds turn out to be
asymptotically weaker than the bounds known from work of Kabatjanski\u\i{} and
Leven\v{s}te\u\i{}n \cite{Kabatjanskii_Levenstein1978:bounds_for_packings}. Our
paper gives more explicit results especially for the underlying modular and
quasimodular functions, in particular we find recurrence relations defining
these functions.

In this paper the dimension $d$ will always be a multiple of $4$.  It is
organised as follows. In Section~\ref{sec:infra} we provide a general study of
functions of the form
\begin{equation}\label{eq:Us-Laplace}
  U(s)=4i\sin\left(\frac\pi2s\right)^2\int_0^{i\infty}\psi(z)e^{i\pi sz}\,dz,\end{equation}
  for a class of functions $\psi$.  
Notice that the integral can be viewed as a Laplace transform after
replacing $z=it$ for $t>0$. We study the analytic continuation of such
functions, which is already given in Viazovska's work
\cite{Viazovska2017:sphere_packing_8}. In Proposition~\ref{prop-Feval} we
formulate conditions on $\psi$ so that the function $U(\|\mathbf{x}\|^2)$ is an
eigenfunction of the Fourier transform. These conditions turn out to be
functional equations for $\psi$ and conditions on the asymptotic behaviour of
$\psi(z)$ for $z\to0$ and $z\to i\infty$. Our main aim is to find the function
$\psi$ so that the last sign change of $U$ is as small as possible. This is
motivated by the choice of functions in
\cite{Cohn-Kumar-Miller+2017:sphere_packing_24,
  Viazovska2017:sphere_packing_8}, as well as by the uncertainty principle
\cite{Bourgain_Clozel_Kahane2010:principe_dheisenberg} mentioned above.

In Section~\ref{sec:even-eigen} we study the set of solutions of the functional
equations given in Section~\ref{sec:infra} for the case of the eigenvalue
$(-1)^{\frac d4}$. We show that the solutions are weakly holomorphic
quasimodular forms of weight $4-\frac d2$ and depth $2$. In order for the
function $U(\|\mathbf{x}\|^2)$ to have the desired properties, we find
conditions on these forms and show that these are uniquely satisfied.

In Section~\ref{sec:odd} the solutions of the functional equations from
Section~\ref{sec:infra} for the case of the eigenvalue $(-1)^{\frac d4+1}$ are
investigated. These turn out to be weakly holomorphic modular forms of weight $2-\frac d2$
for $\Gamma(2)$, a principal congruence subgroup of
$\mathrm{SL}(2,\mathbb{Z})$. Again we characterise the functions $\psi$ that
yield the desired properties for the function $U(\|\mathbf{x}\|^2)$.

In Section~\ref{sec:modular-diff-equat} we find differential equations
satisfied by the forms obtained in Sections~\ref{sec:even-eigen}
and~\ref{sec:odd} and characterise them as certain solutions. Differential
equations turn out to be a convenient method to control vanishing orders of
quasimodular forms. As a byproduct we find linear recurrences for the forms.

In Section~\ref{sec:posit-coeff} we prove that all but possibly finitely many
Fourier coefficients of the quasimodular forms obtained in
Section~\ref{sec:even-eigen} are positive.

In Section~\ref{sec:concluding-remarks} we discuss the modular and
quasimodular forms obtained in Sections~\ref{sec:even-eigen} and~\ref{sec:odd}
for several small dimensions, where the corresponding Fourier eigenfunction
exhibits remarkable behaviour. These cases of course include dimensions $8$,
$12$, and $24$.

In Appendix~\ref{sec:modular} we provide some basic information on modular
functions and forms, as well as quasimodular forms and derivatives of modular
forms, that are needed for Sections~\ref{sec:even-eigen} and~\ref{sec:odd}.

\subsection*{Notation}\label{sec:notation}
Throughout this paper we will use notations that are common in the context of
modular forms. Especially, we denote the two generators of the modular group
$\Gamma=\mathrm{PSL}(2,\mathbb{Z})$ by
\begin{align*}
  &S:z\mapsto -\frac1z\\
  &T:z\mapsto z+1.
\end{align*}
Furthermore, we denote $q=e^{2\pi iz}$, the \emph{nome}, and use a slightly
modified notation for derivatives
\begin{equation*}
f'=\frac1{2\pi i}\frac{df}{dz}=q\frac{df}{dq}.
\end{equation*}
We will freely switch between dependence on $z$ and $q$. The Landau symbol
$\mathcal{O}(q^\ell)$ is always understood for $z\to i\infty$ which is $q\to0$.

In this paper we use the following notation for the Fourier transform of a
function $f\in L^1(\mathbb{R}^d)$:
\begin{equation}
  \label{eq:fourier-def}
  \mathcal{F}(f)(\mathbf{t})=\widehat{f}(\mathbf{t}):=
  \int_{\mathbb{R}^d}f(\mathbf{x})e^{-2\pi
    i\langle\mathbf{x},\mathbf{t}\rangle}\,
  dx_1\cdots dx_d,
\end{equation}
where $\langle\mathbf{x},\mathbf{t}\rangle$ denotes the standard scalar product
in $\mathbb{R}^d$. With this setting we have
\begin{equation*}
  \mathcal{F}\left(e^{\pi i z\|\mathbf{x}\|^2}\right)
  =(-1)^{d/4}z^{-\frac d2}e^{\pi iSz \|\mathbf{t}\|^2}
\end{equation*}
for the Fourier transform of a Gaussian for $\Im z>0$. Here and throughout this
paper $\|\mathbf{x}\|^2=\langle\mathbf{x},\mathbf{x}\rangle$ denotes the
euclidean norm.


\section{Laplace transforms and Fourier
   eigenfunctions}\label{sec:infra}
 In this section we give a general study of functions given in the form
 \eqref{eq:Us-Laplace}. This representation is one of the key ingredient of
 Viazovska's construction of eigenfunctions of the Fourier transform. We
 analyse these functions in some detail and provide their analytic continuation
 to a right half-plane containing the imaginary axis. After these preparations
 we compute the Fourier transform of the function $U(\|\mathbf{x}\|^2)$ and use
 this to obtain necessary and sufficient conditions for this function to be an
 eigenfunction of the Fourier transform (Proposition~\ref{prop-Feval}). This is
 the starting point for the considerations in Sections~\ref{sec:even-eigen}
 and~\ref{sec:odd}. We denote the non-negative imaginary axis by
 $i\mathbb{R}_+:=i(0,\infty)$ and let $L^1_{\rm loc}(i\mathbb{R}_+)$ denote the
 complex valued functions that are absolutely integrable with respect to
 Lebesgue measure on any bounded interval $i(0,b]$.
 \begin{proposition}\label{prop-psiW}
   Suppose $\psi\in L^1_{\rm loc}(i\mathbb{R}_+)$ is such that for some $C>0$
   and constants $a_k$, $b_k\in \mathbb{C}$, $k=0,1,\ldots,n$,
   \begin{equation}\label{eq:psi-infty} \psi(z)=\sum_{k=0}^n a_ke^{-2\pi
       ikz}-iz\sum_{k=0}^n b_k e^{-2\pi ikz} +\mathcal{O}(e^{iCz})
     \quad\text{as }z\to i\infty.
  \end{equation} 
  For $\Re(s)>2n$, let
  \begin{equation}
    \label{eq:W-int}
    W(s):=-i\int_0^{i\infty}\psi(z)e^{i\pi sz}\,dz.
  \end{equation}
Then 
 \begin{equation}\label{eq:W}
 \begin{split}
   &W(s)=\sum_{k=0}^n \left(\frac{a_k}{\pi(s-2k)}+
     \frac{b_k}{\pi^2(s-2k)^2}\right)\\&-i\int_0^{i\infty} \left( \psi (z)-
     \left(\sum_{k=0}^n a_ke^{-2\pi ikz}+z\sum_{k=0}^n b_k e^{-2\pi ikz}
     \right)\right)e^{i\pi sz}\,dz.
  \end{split}
  \end{equation}
  gives an analytic continuation of $W$ to the half-plane  $\Re{(s)}> -C/\pi$.  
  \end{proposition}
 \begin{proof}
   Let $\widetilde W(s)$ be given by the right-hand side of \eqref{eq:W}.  Then
   the local integrability of $\psi$ and the condition \eqref{eq:psi-infty}
   imply that $\widetilde W(s)$ is a well-defined meromorphic function on the
   half plane $\Re(s)>-\frac C\pi$ with (at most) double poles at $s=2k$,
   $k=0,\ldots,n$.  For an integer $k$ and $\Re(s)>2k$, elementary computations
   show
 $$
    -i\int_0^{i\infty} e^{-2\pi ikz}e^{i\pi sz}\,dz = \frac{1}{\pi(s-2k)},   $$
    and
    $$
   \quad -\int_0^{i\infty} z e^{-2\pi ikz} e^{i\pi sz}\,dz
    =\frac{1}{\pi^2(s-2k)^2},
$$
  and hence that $\widetilde W(s)$ agrees with $W(s)$ for $\Re(s)>2n$.
\end{proof}

We next assume that $\psi$ is holomorphic on the upper half-plane.
   
\begin{proposition}\label{prop-psiU}
  Let $\psi:\mathbb{H}\to \mathbb{C}$ be holomorphic on $\mathbb{H}$ and
  bounded on the angular region
  $R_{\alpha,\epsilon}:=\{re^{ it}:0<r<\epsilon, \, \alpha<t<\pi-\alpha\}$ for
  some $\epsilon>0$ and some $0<\alpha<\pi/4$.  Further suppose the restriction
  of $\psi$ to $i\mathbb{R}_+$ and $W$ are as in Proposition~\ref{prop-psiW}
  and for $\Re{(s)}> -C/\pi$ let $U(s)$ be defined by
 \begin{equation}\label{eq:F1}
  U(s):=-4\sin\left(\frac\pi2 s\right)^2W(s).
\end{equation}
Then $U(s)$ is holomorphic for  $\Re{(s)}> -C/\pi$ and 
  \begin{equation}
    \label{eq:U-deformed}
    \begin{split}
         &iU(s) = \int_{-1}^i\psi(Tz)e^{i\pi sz}\,dz 
      +\int_{1}^i\psi(T^{-1}z)e^{i\pi sz}\,dz\\&
       -2\int_0^i\psi(z)e^{i\pi sz}\,dz+
       \int_i^{i\infty}\left(\psi(Tz)-2\psi(z)+\psi(T^{-1}z)\right)
       e^{i\pi sz}\,dz  ,
    \end{split}
    \end{equation}
  where the integrals are along straight line segments joining the endpoints. 
\end{proposition}
\begin{proof} 
  Starting from \eqref{eq:W-int} we derive a second form of the analytic
  continuation of $-4\sin(\frac\pi2s)^2W(s)$, which is more suitable for the
  proof and will also be used later. We write
  \begin{align*}
    iU(s) &=   \int_0^{i\infty}\psi(z)\left(e^{i\pi s(z-1)}-2e^{i\pi sz}+
      e^{i\pi s(z+1)}\right)\,dz\\
    &= 
      \int_{-1}^{-1+i\infty}\psi(Tz)e^{i\pi sz}\,dz -
    2\int_0^{i\infty}\psi(z)e^{i\pi sz}\,dz \\&
    +
       \int_{1}^{1+i\infty}\psi(T^{-1}z)e^{i\pi sz}\,dz,
  \end{align*}
  which follows by expressing the sine in terms of the exponential, expanding
  the square and substituting in the integral. This expression is valid for
  $\Re(s)>2n$. Since $\psi$ is holomorphic on $\mathbb{H}$, bounded on
  $R_{\alpha,\epsilon}$ and satisfies the growth condition
  \eqref{eq:psi-infty}, we may deform the contours of integration as follows:
  the path from $-1$ to $-1+i\infty$ is deformed into a straight line from $-1$
  to $i$ and then along the imaginary axis from $i$ to $i\infty$; similarly,
  the contour from $1$ to $1+i\infty$ is deformed into a straight line from $1$
  to $i$ and then again along the imaginary axis (see
  Figure~\ref{fig:viazovska}).
  \begin{figure}
    \begin{centering}
      \includegraphics[height=6cm]{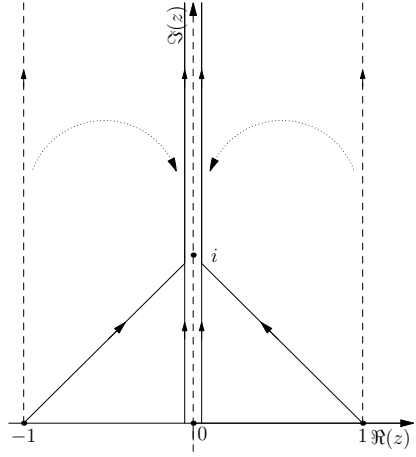}
    \end{centering}
    \caption{Deforming the contour of integration}\label{fig:viazovska}
  \end{figure}

  Collecting terms with matching paths of integration gives
  \eqref{eq:U-deformed} valid for $\Re(s)>2n$.  Since the exponential terms in
  the asymptotic expansion \eqref{eq:psi-infty} for $z\to i\infty$ cancel in
  the last integral, the new expression is also valid for $\Re(s)>-\frac C\pi$
  providing an alternative form for expressing the analytic continuation of
  $U(s)$. The integrals are all absolutely and uniformly convergent for
  $\Re(s)\geq 0$.
\end{proof}

\begin{proposition} \label{prop-FFhat} Let $\psi$ and $U$ be as in
  Proposition~\ref{prop-psiU}, and let $F:\mathbb{R}^d\to\mathbb{C}$ be defined
  by
\begin{equation}\label{Fdef}
F(\mathbf{x}):=U(\|\mathbf{x}\|^2), \qquad (\mathbf{x}\in \mathbb{R}^d).
\end{equation}
If, in addition,  $\psi$  satisfies
\begin{equation}
  \psi(z)=\mathcal{O}(e^{iCSz})\quad\text{as }z\to 0
  \quad\text{non-tangentially in }\mathbb{H},\label{eq:psi-0-exp}
\end{equation}
then $F$ is a Schwartz function and   can be written in the form
 \begin{equation}\label{eq:F-PS}
  \begin{split}
    F&(\mathbf{x}) =
 -i\left[\int_{-1}^i\psi(Tz)e^{i\pi \|\mathbf{x}\|^2z}\,dz \right.
      +\int_{1}^i\psi(T^{-1}z)e^{i\pi \|\mathbf{x}\|^2z}\,dz\\&
       \left.-2\int_0^i\psi(z)e^{i\pi \|\mathbf{x}\|^2z}\,dz+
     \int_i^{i\infty}\!\!\!\!\left(\psi(Tz)-2\psi(z)+\psi(T^{-1}z)\right)e^{i\pi
          \|\mathbf{x}\|^2z}\,dz  \right].
    \end{split}
  \end{equation}
Consequently, the Fourier transform of $F$ is given by
  \begin{equation}\label{eq:Fhat-PS}
  \begin{split}
    & \hat{F} (\mathbf{t}) =
 -i(-1)^{d/4}\left[\int_{-1}^i \psi(T^{-1}Sz)e^{i\pi \|\mathbf{t}\|^2z} z^{d/2-2}\,
   dz\right.
 \\&   +2\int_i^{i\infty}\psi(Sz)e^{i\pi \|\mathbf{t}\|^2z} z^{d/2-2}\, dz+
 \int_{1}^i \psi(TSz)e^{i\pi \|\mathbf{t}\|^2z} z^{d/2-2}\, dz \\
 &\left.-\int_{0}^i (\psi(T^{-1}Sz)-2\psi(Sz)+\psi(TSz))e^{i\pi \|\mathbf{t}\|^2z}
 z^{d/2-2}\, dz \right].  \end{split}
  \end{equation}
\end{proposition}
\begin{proof}
  The representation \eqref{eq:F-PS} follows immediately from the definition
  \eqref{Fdef} and the relation \eqref{eq:U-deformed} of
  Proposition~\ref{prop-psiU}.  Condition \eqref{eq:psi-0-exp} implies that
  $\psi$ vanishes to arbitrary order at $z=0$.  Hence, using \eqref{eq:W-int}
  it follows using well known properties of the Laplace transform
  (see~\cite{Widder1941:laplace_transform}) that $F$ and its derivatives all
  decay faster than any negative power of $\|\mathbf{x}\|$. Since $U$ is
  analytic, it follows that $F$ is a Schwartz function.
  
Thus we can compute the Fourier transform of $F$ by Fubini's theorem
  \begin{align*}
    \widehat{F}(\mathbf{t})=
    &-i\left[
\int_{-1}^i\psi(Tz)e^{i\pi\|\mathbf{t}\|^2Sz}(-iz)^{-\frac d2}\,dz\right.\\
      &+\int_{1}^i\!\!\psi(T^{-1}z)e^{i\pi\|\mathbf{t}\|^2Sz}(-iz)^{-\frac d2}\,dz
       -2\int_0^i\!\!\psi(z)e^{i\pi\|\mathbf{t}\|^2Sz}(-iz)^{-\frac d2}\,dz\\&+
       \left. \int_i^{i\infty}\left(\psi(Tz)-2\psi(z)+\psi(T^{-1}z)\right)
         e^{i\pi\|\mathbf{t}\|^2Sz}(-iz)^{-\frac d2}\,dz
    \right].
  \end{align*}
Substituting $Sz$ in this expression and collecting signs gives
\eqref{eq:Fhat-PS}.
\end{proof}

\begin{proposition}\label{prop-Feval}
Let $\psi$  and  $F$ be as in
Proposition~\ref{prop-FFhat} and $\varepsilon\in \{-1,1\}$. 
Then 
 $\hat{F}  =\varepsilon(-1)^{\frac d4} F$ if and only if
 \begin{align}
     z^{\frac d2-2}\psi(T^{-1}Sz)&=\varepsilon\psi(Tz)\label{eq:psi1}\\
2z^{\frac d2-2}\psi(Sz)&=\varepsilon\left(\psi(Tz)-2\psi(z)+\psi(T^{-1}z)\right),
\label{eq:psi2} 
 \end{align}
 for all $z\in \mathbb{H}$.
\end{proposition}
\begin{proof}
  We consider the auxiliary function
  \begin{equation}
    \label{eq:aux}
    H(s)=\sum_{\ell=1}^4I_\ell(s),
  \end{equation}
  where
  \begin{equation}
    \label{eq:I1-I4}
    \begin{split}
      I_1(s)&=\int_{i}^{i\infty}\left(\psi(Tz)-2\psi(z)+\psi(T^{-1}z)-
        2\varepsilon z^{d/2-2}\psi(Sz))\right)e^{i\pi sz}\,dz\\
      I_2(s)&=\int_0^i\left(\varepsilon
  \left(\psi(TSz)-2\psi(Sz)+\psi(T^{-1}Sz)\right)z^{d/2-2}-2\psi(z)\right)e^{i\pi
  sz}\,dz\\
I_3(s)&=\int_{-1}^i\left(\psi(Tz)-\varepsilon\psi(T^{-1}Sz)z^{d/2-2}\right)
e^{i\pi sz}\,dz\\
I_4(s)&=\int_{1}^i\left(\psi(T^{-1}z)-\varepsilon\psi(TSz)z^{d/2-2}\right)
e^{i\pi sz}\,dz.
    \end{split}
  \end{equation}
  Then
  $H(\|\mathbf{x}\|^2)=F(\mathbf{x})-\varepsilon(-1)^{\frac
    d4}\widehat{F}(\mathbf{x})$. If $F$ is an eigenfunction for the Fourier
  transform with eigenvalue $\varepsilon(-1)^{\frac d4}$, then $H(s)$ vanishes
  on the positive real axis and thus for all $s\in\mathbb{C}$.

  Now the functions $I_2,I_3,I_4$ are entire functions of exponential type
  $\pi$ (see~\cite{Boas1954:entire_functions}). Thus for $H(s)$ to vanish
  identically, $I_1(s)$ also has to be of exponential type. By the Paley-Wiener
  theorem (see\cite{Paley_Wiener1934:fourier_transformations}) this implies
  that the integrand defining $I_1$ has to have compact support. Since the
  integrand is analytic it has to vanish identically giving
    \begin{equation}
    \label{eq:psi2b}
    2z^{\frac d2-2}\psi(Sz)=\varepsilon
    \left(\psi(Tz)-2\psi(z)+\psi(T^{-1}z)\right).
  \end{equation}
Since the integrand in the definition of $I_2$ equals the
  integrand for $I_1$ after a substitution $z\mapsto Sz$, also $I_2$ vanishes
  identically.

  Now it remains $H(s)=I_3(s)+I_4(s)=0$. We observe that
  \begin{equation}
    \label{eq:lim}
    \begin{split}
      &\lim_{t\to-\infty}I_3(t)=\infty,\quad\lim_{t\to\infty}I_3(t)=0\\
      &\lim_{t\to+\infty}I_4(t)=\infty,\quad\lim_{t\to-\infty}I_4(t)=0.
    \end{split}
  \end{equation}
  Thus $H(s)=0$ can only hold, if $I_3$ and $I_4$ vanish identically. Again
  both integrals $I_3$ and $I_4$ can be rewritten as Laplace transforms after a
  change of variables. It follows that
    \begin{align}
    z^{\frac d2-2}\psi(T^{-1}Sz)&=  \varepsilon\psi(Tz)\label{eq:psi1a}\\
   z^{\frac d2-2}\psi(TSz)&= \varepsilon\psi(T^{-1}z)\label{eq:psi1b}.
 \end{align}
 
 It is immediate that \eqref{eq:psi1a} and \eqref{eq:psi1b} are equivalent by
 substituting $z\mapsto Sz$.
\end{proof}

\section{The $(-1)^{\frac d4}$ eigenfunction}
\label{sec:even-eigen}
In this section we first prove that the function $\psi$ in \eqref{eq:W-int} has
to be a quasimodular form in order to make $F$ given by \eqref{eq:F-PS} an
eigenfunction of the Fourier transform for the eigenvalue $(-1)^{\frac
  d4}$. The main ingredient for this fact is Proposition~\ref{prop-even} below,
which is essentially equivalent to
\cite[Proposition~4.7]{Cohn_Kumar_Miller+2019:universal_optimality} except for
different growth conditions on the function $\psi$. We developed this result
independently, and keep the proof in order to make the paper self-contained. In
Section~\ref{sec:determining-psi} we then use dimension arguments for the
underlying spaces of quasimodular forms to achieve the existence of such forms
that are useful for the construction of eigenfunctions with certain extremal
properties.

\begin{proposition}\label{prop-even}
  Let $\psi$ be as in Proposition~\ref{prop-psiU}. Then the corresponding
  function $F$ given by \eqref{Fdef} is an eigenfunction for the Fourier
  transform with eigenvalue $(-1)^{\frac d4}$, if and only if
  $z^{\frac d2-2}\psi(Sz)$ is a weakly holomorphic quasimodular form of weight $4-\frac d2$ and
  depth $2$. More precisely, there are weakly holomorphic modular forms
  $\psi_1$, $\psi_2$, $\psi_3$ of respective weights $4-\frac d2$,
  $2-\frac d2$, and $-\frac d2$ such that
  \begin{equation}
    \label{eq:psi-quasi}
    z^{\frac d2-2}\psi(Sz)=\psi_1(z)-2E_2(z)\psi_2(z)+E_2(z)^2\psi_3(z).
  \end{equation}
  This gives
  \begin{equation}
    \label{eq:psi-quasi-expl}
    \begin{split}
      \psi(z)&=z^2\left(\psi_1(z)-2E_2(z)\psi_2(z)+E_2(z)^2\psi_3(z)\right)\\
      &+ z\frac{12i}{\pi}\left(\psi_2(z)-E_2(z)\psi_3(z)\right)-
      \frac{36}{\pi^2}\psi_3(z).
    \end{split}
  \end{equation}
  Furthermore, $\psi_1$, $\psi_2$, and $\psi_3$ have to satisfy
  \begin{equation}
    \label{eq:psi-infty-even}
    \psi_1(z)-2E_2(z)\psi_2(z)+E_2(z)^2\psi_3(z)=\mathcal{O}(e^{2\pi iz})
  \end{equation}
for $z\to i\infty$ in order to fulfil \eqref{eq:psi-infty} and
\eqref{eq:psi-0-exp}.
\end{proposition}
\begin{proof}
  By Proposition~\ref{prop-Feval} a function $F$ given in the form
  \eqref{eq:F1} is an eigenfunction of the Fourier transform for the eigenvalue
  $(-1)^{\frac d4}$ (this is $\varepsilon=1$) if and only if \eqref{eq:psi1}
  and \eqref{eq:psi2} hold.  From \eqref{eq:psi1} we obtain
  \begin{equation*}
    \psi(z)=(z+1)^{\frac d2-2}\psi(TSTz)
  \end{equation*}
  and then
  \begin{equation*}
    (z+1)^{\frac d2-2}\psi(STz)=z^{\frac d2-2}\psi(TSTSTz)=z^{\frac d2-2}\psi(Sz),
  \end{equation*}
  where we have used that $(TS)^3=\mathrm{id}$. Thus the function
  \begin{equation*}
    \phi(z)=z^{\frac d2-2}\psi(Sz)
  \end{equation*}
  is periodic with period $1$.

  Now we write \eqref{eq:psi2} as
  \begin{equation}\label{eq:psi4a}
    \psi(Tz)-2\psi(z)+\psi(T^{-1}z)=2\phi(z)
  \end{equation}
  and set
  \begin{equation}\label{eq:fdef}
    f(z)=\psi(Tz)-\psi(z)-(2z+1)\phi(z).
  \end{equation}
  Then we have
  \begin{multline*}
    f(z)-f(T^{-1}z)\\=
    \psi(Tz)-2\psi(z)+\psi(T^{-1}z)-(2z+1)\phi(z)+(2z-1)\phi(T^{-1}z).
  \end{multline*}
  Using the periodicity of $\phi$ and \eqref{eq:psi4a} gives the periodicity of
  $f$. Now we set
  \begin{equation}\label{eq:gdef}
    g(z)=\psi(z)-z^2\phi(z)-zf(z).
  \end{equation}
  We compute
  \begin{equation*}
    g(Tz)-g(z)=\psi(Tz)-\psi(z)-((z+1)^2-z^2)\phi(z)-((z+1)-z)f(z)=0,
  \end{equation*}
  where we have used the periodicity of $\phi$ and $f$ as well as the
  definition of $f$. This shows that also $g$ is periodic.

  Thus $\psi$ satisfies the relation
  \begin{equation}
    \label{eq:psi-quasi1}
    \psi(z)=z^{\frac d2}\psi(Sz)+zf(z)+g(z)
  \end{equation}
  for two (yet unknown) periodic functions $f$ and $g$. We now use definitions
  \eqref{eq:fdef} and \eqref{eq:gdef} to express $g$ in terms of $\psi$
  \begin{equation}\label{eq:gpsi}
    g(z)=(z+1)\psi(z)-z\psi(Tz)+z(z+1)z^{\frac d2-2}\psi(Sz).
  \end{equation}
  Substituting $STz$ and multiplying through the denominator yields
  \begin{equation}
    \label{eq:gST}
    \begin{split}
      (z+1)^{\frac d2}g(STz)&=z(z+1)(z+1)^{\frac d2-2}\psi(STz)\\&+
      (z+1)(z+1)^{\frac d2-2}\psi(ST^{-1}Sz)-z\psi(Tz),
    \end{split}
  \end{equation}
  where we have used $TST=ST^{-1}S$. We have already established the
  periodicity of $\phi(z)=z^{\frac d2-2}\psi(Sz)$. This allows to replace the
  first and the second term to yield
  \begin{equation*}
    g(STz)=(z+1)z^{\frac d2-1}\psi(Sz)+(z+1)\psi(z)-z\psi(Tz)=g(z).
  \end{equation*}
  Using the already established periodicity of $g$ this gives
  \begin{equation}
    \label{eq:gS}
    z^{\frac d2}g(Sz)=g(z);
  \end{equation}
  together with the growth condition \eqref{eq:psi-infty} this shows that $g$
  is a weakly holomorphic modular form of weight $-\frac d2$.

  Applying $z\mapsto Sz$ to \eqref{eq:psi-quasi1} and adding this to
  \eqref{eq:psi-quasi1} (divided by $z^{\frac d2}$) yields
  \begin{equation}
    \label{eq:f-mod}
    z^{\frac d2-2}f(Sz)=f(z)+\frac2zg(z);
  \end{equation}
  $f$ is a weakly holomorphic quasimodular form of weight $2-\frac d2$ and
  depth $1$ (again using \eqref{eq:psi-infty}).

  We set
  \begin{equation*}
    h(z)=f(z)-\frac{\pi i}3E_2(z)g(z)
  \end{equation*}
  and use $z^{-2}E_2(z)=E_2(z)-\frac{6i}{\pi z}$ to obtain
  \begin{equation*}
    z^{\frac d2-2}h(Sz)=h(z),
  \end{equation*}
  which together with the obvious periodicity and the growth
  condition~\eqref{eq:psi-infty} yields that $h$ is a weakly holomorphic
  modular form of weight $2-\frac d2$. Inserting this into
  \eqref{eq:psi-quasi1} gives the quasimodularity of $z^{\frac d2-2}\psi(Sz)$
  with weight $4-\frac d2$ and depth $2$. By the structure theorem of
  quasimodular forms (see \cite{Royer2012:quasimodular_forms_introduction,
    Zagier2008:elliptic_modular_forms,
    Kaneko_Zagier1995:generalized_jacobi_theta}), $z^{\frac d2-2}\psi(Sz)$ can
  then be written as \eqref{eq:psi-quasi}, where we have set
  \begin{align*}
    \psi_1(z)&=z^{\frac d2-2}\psi(Sz)-E_2(z)h(z)-E_2(z)^2g(z)\\
    \psi_2(z)&=-\frac{\pi i}{12}h(z)\\
    \psi_3(z)&=-\frac{\pi^2}{36}g(z).
  \end{align*}

  In order to satisfy condition \eqref{eq:psi-infty}, the term multiplied by
  $z^2$ in \eqref{eq:psi-quasi-expl} has to tend to $0$ for $z\to i\infty$,
  which gives \eqref{eq:psi-infty-even}. By \eqref{eq:psi-quasi} this implies
  that \eqref{eq:psi-infty} and \eqref{eq:psi-0-exp} are satisfied for any
  $0<C<2\pi$.
\end{proof}

\subsection{Determining $\psi$}\label{sec:determining-psi}
In a next step we want to determine $\psi$ (or equivalently
$\psi_1$, $\psi_2$, $\psi_3$) to satisfy \eqref{eq:psi-infty-even}. Since
$\psi_1$, $\psi_2$, and $\psi_3$ are weakly holomorphic modular forms of
respective weights $4-\frac d2$, $2-\frac d2$, and $-\frac d2$, we 
express these forms as
\begin{equation}
  \label{eq:psi-ansatz}
  \begin{split}
    \psi_1&=\frac1{\Delta^\ell}\omega_{k+2} P_n^{(k)}(j)\\
    \psi_2&=\frac1{\Delta^\ell}\omega_{k+1} Q_n^{(k)}(j)\\
    \psi_3&=\frac1{\Delta^\ell}\omega_k R_n^{(k)}(j),
  \end{split}
\end{equation}
for $\ell\in\mathbb{N}$ chosen so that $\psi_m\Delta^\ell$ ($m=1,2,3$) are
weakly holomorphic modular forms of positive weight; $P_n^{(k)}$, $Q_n^{(k)}$,
and $R_n^{(k)}$ are polynomials, which have to be determined. The parameter $n$
is an ordering parameter related to the degrees of these polynomials. 
It will be determined by \eqref{eq:nmin} in the course of the
following discussion.

The minimal possible choice of $\ell$ is then
\begin{equation*}
  \ell=\left\lceil\frac d{24}\right\rceil.
\end{equation*}
Furthermore, we set
\begin{equation*}
  k=6\ell-\frac d4,
\end{equation*}
which gives $0\leq k\leq5$.  The forms $\omega_m$ in \eqref{eq:psi-ansatz} are
modular forms of weight $2m$ ($m=0,\ldots,7$), which are given in
Table~\ref{tab:omega}; these forms are uniquely determined by the requirement
to be holomorphic, or to have a pole of minimal order at $i\infty$.
The parameter $n$ refers to the order of the pole of
$\omega_{k+2} P_n^{(k)}(j)$, $\omega_{k+1} Q_n^{(k)}(j)$, or
$\omega_k R_n^{(k)}(j)$. Notice that for $m=1$ the form $\omega_m$ has a simple
pole at $i\infty$, whereas for $m=6,7$ it has a simple zero there. This affects
the possible degrees of the polynomials $P_n^{(k)}$, $Q_n^{(k)}$, and
$R_n^{(k)}$, see Table~\ref{tab:PQR}. This table also gives the dimension of
the space $\mathcal{Q}_n^{(2k+2)}$ of weakly holomorphic quasimodular forms of
weight $2k+2$ and depth $2$, which have a pole of order at most $n$ at
$i\infty$. The table also gives the definition of the quantity $a(k)$, which
will be needed in the sequel.

\begin{table}[h]
  \setlength{\tabulinesep}{1mm}
  \begin{tabu}[h]{|l|c|}
    \hline
    $m$&$\omega_m$\\\hline
    $0$&1\\\hline
    $1$&$-j'=\frac{E_4^2E_6}{\Delta}=\frac{E_6}{E_4}j$\\\hline
    $2$&$E_4$\\\hline
    $3$&$E_6$\\\hline
    $4$&$E_4^2$\\\hline
    $5$&$E_4E_6$\\\hline
    $6$&$\Delta=\Delta\omega_0$\\\hline
    $7$&$E_4^2E_6=\Delta\omega_1$\\\hline
  \end{tabu}
  \medskip
  \caption{The choices of $\omega_m$}\label{tab:omega}
\end{table}

\begin{table}[h]
  \setlength{\tabulinesep}{1mm}
  \centering
  \begin{tabu}[h]{|l|c|c|c|c|c|}
    \hline
    $k$&$\deg P_n^{(k)}$&$\deg Q_n^{(k)}$
    &$\deg R_n^{(k)}$&$\dim\mathcal{Q}_n^{(2k+2)}$&$a(k)$\\
    \hline
    $0$&$n$&$n-1$&$n$&$3n+2$&$1$\\
    \hline
    $1$&$n$&$n$&$n-1$&$3n+2$&$1$\\
    \hline
    $2$&$n$&$n$&$n$&$3n+3$&$2$\\
    \hline
    $3$&$n$&$n$&$n$&$3n+3$&$2$\\
    \hline
    $4$&$n+1$&$n$&$n$&$3n+4$&$3$\\
    \hline
    $5$&$n$&$n+1$&$n$&$3n+4$&$3$\\
    \hline
  \end{tabu}
  \medskip
  \caption{Degrees of the polynomials $P_n^{(k)}$,
    $Q_n^{(k)}$, and $R_n^{(k)}$}\label{tab:PQR}
\end{table}

In light of \eqref{eq:psi-quasi-expl} and the asymptotic behaviour of $\psi$
\eqref{eq:psi-infty} used in Proposition~\ref{prop-psiW} we require that the
polar order of $\psi_2(z)-E_2(z)\psi_3(z)$ (the term multiplied by $z$ in
\eqref{eq:psi-quasi-expl}) is $1$ less than the polar order of
$\psi_3(z)$. This ensures by \eqref{eq:W} that the largest real second order
pole of $W(s)$ is $2$ less than the largest real first order pole. Notice that
condition \eqref{eq:psi-infty-even} ensures that $W(s)$ has no third order
poles in the right half plane. Together this ensures that the polar order of
$\psi$ at $i\infty$ corresponds to the desired sign change of the function $F$
given by \eqref{Fdef}.

In order to achieve the behaviour of $\psi$ described in the last paragraph, we
use the degrees of freedom given by $\dim\mathcal{Q}_n^{(2k+2)}$ to first
ensure that
\begin{equation}
  \label{eq:f-ord}
  \omega_{k+1}Q_n^{(k)}(j)-E_2\omega_kR_n^{(k)}(j)=\mathcal{O}(q^{-n+1})
\end{equation}
and second to eliminate as many Laurent series coefficients of
\begin{equation*}
\omega_{k+2}P_n^{(k)}(j)-2E_2\omega_{k+1}Q_n^{(k)}(j)+E_2^2\omega_kR_n^{(k)}(j)
\end{equation*}
as possible. By solving the according linear equations we can achieve
\begin{equation}
  \label{eq:psi-order}
  \omega_{k+2}P_n^{(k)}(j)-2E_2\omega_{k+1}Q_n^{(k)}(j)+E_2^2\omega_kR_n^{(k)}(j)=
  \mathcal{O}(q^{2n+a(k)-1}).
\end{equation}
In order for $\psi$ to satisfy \eqref{eq:psi-infty-even} we have to choose $n$
so that
\begin{equation*}
  2n+a(k)-1>\ell;
\end{equation*}
the minimal possible choice for $n$ is then
\begin{equation}\label{eq:nmin}
  n=\left\lceil\frac{\ell-a(k)+2}2\right\rceil.
\end{equation}

Condition \eqref{eq:f-ord} ensures that there is a sign change of
$F(\mathbf{x})$ at $\|\mathbf{x}\|^2=2n+2\ell$ and $F(\mathbf{x})\neq0$ for
$\|\mathbf{x}\|^2=2n+2\ell-2$. Expressing $\ell$, $k$, and $n$ in terms of $d$
yields $2n+2\ell=2\lfloor\frac{d+4}{16}\rfloor+2$.

Summing up, we have proved most of the following theorem. For the sake of
simplicity, we abuse notation by writing $f(\mathbf{x})=f(\|\mathbf{x}\|)$,
whenever $f$ is a radial function and the context is clear.
\begin{theorem}\label{thm-even}
  For $d\equiv0\pmod4$ set $n_+=\lfloor\frac{d+4}{16}\rfloor+1$. Then there
  exists a radial Schwartz function $F_+:\mathbb{R}^d\to\mathbb{R}$ satisfying
  \begin{equation}\label{eq:f-plus}
    \begin{split}
      F_+(\mathbf{x})&=(-1)^{\frac d4}\widehat{F}_+(\mathbf{x})
      \quad\text{for all }\mathbf{x}\in\mathbb{R}^d\\
   F_+(\sqrt{2n_+})&=0\quad\text{and }F_+'(\sqrt{2n_+})\neq0\\
   F_+(\sqrt{2m})&=F'_+(\sqrt{2m})=0\quad\text{for }m>n_+,\quad m\in\mathbb{N}.
    \end{split}
  \end{equation}
  If Conjecture~\ref{conj1} stated below holds, then $\sqrt{2n_+}$ is the last
  sign change of the function $F_+$. This is the case for all dimensions
  $\leq312$ by Remark~\ref{rem1} stated below.
\end{theorem}
\begin{proof}
  It only remains to show that $F_+'(\sqrt{2n_+})\neq0$. This follows from the
  discussion of the vanishing orders in Section~\ref{sec:diff-equat-quasi},
  especially equation~\eqref{eq:hw-ord} and the explanation following. More
  precisely, \eqref{eq:hw-ord} implies that $a_{n_+}\neq0$ in
  \eqref{eq:psi-infty}, whereas \eqref{eq:gw-ord} implies that
  $b_{n_+}=0$. This shows by \eqref{eq:W} that $W$ has a simple pole at
  $s=2n_+$, which gives a simple zero of $U$ at this point.
\end{proof}
\begin{remark}\label{rem:extra-even}
  Notice that by \eqref{eq:psi-order} the order increases by $2$, if $n$
  increases by $1$. Together with \eqref{eq:nmin} this gives that for certain
  $d\pmod{48}$ the order of vanishing of $\psi$ is one more than required. This
  means that for these values of the dimension there is one extra degree of
  freedom, that can be used for instance to require $F(\mathbf{0})=0$. This is
  the case for dimensions
  \begin{equation}
    \label{eq:extra-even}
    d\equiv\{0,12,16,28,32,44\}\pmod{48}.
  \end{equation}
\end{remark}
\begin{table}
  \setlength{\tabulinesep}{1mm}
  \begin{tabu}[h]{|l|>{\raggedleft}p{6mm}|>{\raggedleft}p{6mm}|>{\raggedleft}p{5mm}|>{\raggedright}p{2.2cm}|>{\raggedright}p{2.2cm}|>{\raggedright}p{2.2cm}|}
    \cline{2-7}
    \multicolumn{1}{c}{}&\multicolumn{2}{|c|}{$d$}&$n$&\small$P_n^{(k)}(w)$&\small$Q_n^{(k)}(w)$&\small$R_n^{(k)}(w)$\\
    \hline
    \multirow{2}{*}[2mm]{$k=0$}&&$24$&$1$&\tiny$w-3528$&\tiny $1$&\tiny $w+1800$\\
    \cline{2-7} &$48$&$72$&$2$&\tiny$175w^2-1840638 w$ $-475793136$
    &\tiny$175 w+497922$
    &\tiny$175w^2+2534082 w$ $+111078000$\\
    \hline \hline
    \multirow{2}{*}[2mm]{$k=1$}  &  &$20$&$1$&\tiny$w-1008$&\tiny $w-1368$ &\tiny$1$\\
    \cline{2-7} &$44$&$68$&$2$&\tiny$25 w^2-167286 w$ $-10456992$
    &\tiny$25 w^2-18966 w$ $-41044752$
    &\tiny$25 w+172554$\\
    \hline \hline
    \multirow{3}{*}[5mm]{$k=2$}  &  &&$0$&\tiny$1$&\tiny$1$&\tiny$1$\\
    \cline{2-7}
    &$16$&$40$&$1$&\tiny$w-5628$&\tiny$w+420$ &\tiny$w+4740$\\
    \cline{2-7} &$64$&$88$&$2$&\tiny$21 w^2-277373 w$ $-147949620$
    &\tiny$21 w^2+104155 w$ $+2942940$
    &\tiny$21 w^2+449395 w$ $+62398380$\\
    \hline \hline
    \multirow{3}{*}[5mm]{$k=3$}  &    &&$0$&\tiny$1$&\tiny$1$&\tiny$1$\\
    \cline{2-7}
    & $12$&$36$&$1$&\tiny$w-2548$ &\tiny $w-1588$ &\tiny $w+1100$ \\
    \cline{2-7} & $60$&$84$&$2$&\tiny$7 w^2-63953 w$ $-13216476$&\tiny
    $7 w^2+3079 w$ $-26138316$
    &\tiny$7 w^2+82207 w$ $+2838660$\\
    \hline \hline
    \multirow{2}{*}[2mm]{$k=4$}  &    &$8$&$0$&\tiny$ w-1728$&\tiny$1$&\tiny$1$\\
    \cline{2-7}
    &$32$&$56$&$1$&\tiny$5 w^2-39879 w-3302208$&\tiny $5 w+6741$ &\tiny $5 w+44721$\\
    \hline \hline
    \multirow{2}{*}[2mm]{$k=5$}  &    &$4$&$0$&\tiny$1$&\tiny $w-864$ &\tiny$1$\\
    \cline{2-7}
    &$28$&$52$&$1$&\tiny $w-4473$ &\tiny $w^2-1413 w-453600$ &\tiny $w+3375$\\
    \hline
  \end{tabu}
  \caption{The polynomials $P_n^{(k)}$, $Q_n^{(k)}$, and $R_n^{(k)}$;  the
    dimensions in the left column are those covered by
    Remark~\ref{rem:extra-even}.}
  \label{tab:PQR1}
\end{table}


\section{Eigenfunctions for eigenvalue
  $(-1)^{\frac{d}{4}+1}$}\label{sec:odd}
In this section we consider eigenfunctions of the Fourier transform with
eigenvalue $(-1)^{\frac{d}{4}+1}$ of the form \eqref{eq:Us-Laplace}. We show in
Proposition~\ref{prop4.1} that, in this case, the function $\psi$ can be
expressed in terms of weakly holomorphic modular forms for $\Gamma$ and
$\Gamma(2)$. This proposition is essentially equivalent to
\cite[Proposition~4.8]{Cohn_Kumar_Miller+2019:universal_optimality}, except for
different growth conditions in the assumptions. This result was found
independently; we keep the proof in order to keep the presentation
self-contained. We then explore explicit
representations of these forms and show the existence of eigenfunctions
satisfying similar extremal properties as in Section~\ref{sec:even-eigen}.

In the following the function $\llambda$ ($\lambda$ being the Hauptmodul for
$\Gamma(2)$) will play an important role. Properties of this function are
discussed in Appendix~\ref{sec:modular}.
\begin{proposition}\label{prop4.1}
  Let $\psi$ be as in Proposition~\ref{prop-psiU}. Then the corresponding
  function $F$ given by \eqref{Fdef} is an eigenfunction of the Fourier
  transform with eigenvalue $(-1)^{\frac{d}{4}+1}$ if and only if there exists
  a weakly holomorphic modular form $f$ of weight $2-\frac{d}{2}$ for $\Gamma$
  and $\omega$ a weakly holomorphic modular form of weight $2-\frac{d}{2}$ for
  $\Gamma(2)$ such that
\begin{align}
 \psi(z) & = f(z)\cdot \log\lambda(z) + \omega(z),\label{eq:psi-sol} \\
\omega(z) & = z^{\frac{d}{2}-2}\omega(Sz) + \omega(Tz),\label{eq:omega-func}
\end{align} 
where $\log\lambda$ is defined in \eqref{logexp}. 
\end{proposition} 
\begin{proof}
  By Proposition~\ref{prop-Feval} with $\epsilon = -1$, $F$ is an eigenfunction
  of the Fourier transform with eigenvalue $(-1)^{\frac{d}{4}+1}$ iff $\psi$
  satisfies the two equations:
  \begin{align}
\label{eqn1} z^{\frac{d}{2}-2}\psi(TSz) & = -\psi(T^{-1}z),\\
\label{eqn2} 2z^{\frac{d}{2}-2}\psi(Sz) & =
-(\psi(Tz) -2\psi(z) + \psi(T^{-1}z)).
\end{align} 
To solve these we first consider $H(z) : = z^{\frac{d}{2}-2}\psi(Sz)$ which by
\eqref{eqn1} gives
\begin{equation} \begin{split}
H(Tz) & = (Tz)^{\frac{d}{2}-2}\psi(STz) = (Tz)^{\frac{d}{2}-2}\psi(T^{-1}TSTz)\\
          &= -(Tz)^{\frac{d}{2}-2}(TSTz)^{\frac{d}{2}-2}\psi(TSTSTz)\\
          &= -z^{\frac{d}{2}-2}\psi(Sz)= -H(z),
          \end{split}
\end{equation} 
Where we used the property $(TS)^{3} = \mathrm{id}$ in the middle
line. Iterating this property once gives that $H(z+2) = H(z)$ and unravelling
this statement in terms of $\psi$ gives \begin{equation}
  (2z-1)^{\frac{d}{2}-2}\psi(ST^2Sz) = \psi(z).
\end{equation}
Substituting $z \rightarrow STz$ in  \eqref{eqn2} and applying \eqref{eqn1}
repeatedly to get
 \begin{equation} \begin{split}
     2\psi(Tz) & = -(Tz)^{\frac{d}{2}-2}\left(\psi(T^{-1}STz) - 2\psi(STz) +
       \psi(TSTz)\right) \\
 & = \psi(T^{2}z) +2(Tz)^{\frac{d}{2}-2}\psi(STz) + \psi(z)\\
 & = \psi(T^{2}z) -2z^{\frac{d}{2}-2}\psi(Sz)+\psi(z)\\
 & = \psi(T^{2}z)  + \psi(Tz) -\psi(z) + \psi(T^{-1}z).
\end{split}
\end{equation} 
So, altogether we have that
$\psi(T^{2}z) -\psi(Tz) -\psi(z) + \psi(T^{-1}z) = 0$. Defining
$G(z) := \psi(Tz) - \psi(T^{-1}z)$ implies $G(z+1) = G(z)$. Furthermore using
\eqref{eqn1} we obtain
\begin{equation} \begin{split}
    z^{\frac{d}{2}-2}G(Sz) & = z^{\frac{d}{2}-2}(\psi(TSz) - \psi(T^{-1}Sz))\\
    & = -\psi(T^{-1}z) + \psi(Tz)\\
    & = G(z).
  \end{split}
\end{equation} 
Therefore, $G$ is weakly holomorphic modular of weight $2-\frac{d}{2}$ for the
full modular group using the growth condition \eqref{eq:psi-infty}.
We now define
\begin{equation}
  \omega(z) := \psi(z) - \frac{1}{\pi i}G(z)\cdot \log\lambda(z)
\end{equation} 
and from \eqref{eq:log-lambda} we see that $\omega$ is a weakly holomorphic
modular form of weight $2-\frac{d}{2}$ for $\Gamma(2)$ (again using
\eqref{eq:psi-infty}). Moreover, plugging this relationship into \eqref{eqn1}
gives
\begin{equation} \label{eq-omega} \omega(z) =
  z^{\frac{d}{2}-2}\omega(Sz) + \omega(Tz).
\end{equation}
Finally, setting $f : = \frac{1}{\pi i} \cdot G$ we get the desired
conclusions.
\end{proof}
\subsection{Determining $\psi$}\label{sec:determining-psi-odd}
In this step our goal will be determining $\psi$ given its representation in
terms of $f$ and $\omega$. We use the fact that $\mathbb{C}(\lambda)$ is a
field extension of $\mathbb{C}(j)$ to characterise the solutions of
\eqref{eq-omega}. Then using linear algebra, we ensure that conditions
\eqref{eq:psi-infty} and \eqref{eq:psi-0-exp} hold. We will show that due to
\eqref{eq-omega}, achieving the former condition will give the latter.

To begin, we recall $f$ and $\omega$ are weakly holomorphic modular forms of
weight $2-\frac{d}{2}$ for the groups $\Gamma$ and $\Gamma(2)$
respectively. There are no modular forms of negative weight because such forms
must have poles on either $\mathbb{H}$ or at the cusps. The contour integration
arguments from Proposition~\ref{prop-psiU} rule out the former and so $f$ and
$\omega$ must and can only have poles at the cusps. To continue,
define
\begin{align*}
  \ell & = \left \lceil \frac{d-4}{24} \right \rceil \\
  k & = 6\ell - \frac{d-4}{4},
\end{align*}
which gives $0 \leq k \leq 5$. From this we set
\begin{align}
f & = \frac{\omega_{k}}{\Delta^{\ell}}P^{(k)}(j) \\
\label{romega}\omega & = \frac{\omega_{k}}{\Delta^{\ell}}R^{(k)}(\lambda),
\end{align} 
where we recall from Table~\ref{tab:omega} that $\omega_{k}$ is a weakly
holomorphic modular form for the full modular group of weight $2k$.  $P^{(k)}$
is a polynomial associated with each $k$, and $R^{(k)}$ is a rational function
depending on our choice of $k$. This representation follows because
$f \cdot \frac{\Delta^{\ell}}{\omega_{k}}$ is a weakly holomorphic form of
weight 0 and using the fact that $j$ is Hauptmodul for $\Gamma$, this implies
that it must be rational function in $j$. Moreover, since such a rational
function can only have poles at $\infty$, it must be a
polynomial. Analogously, since $\lambda$ is Hauptmodul for $\Gamma(2)$ we can
similarly conclude that $\omega \cdot \frac{\Delta^{\ell}}{\omega_{k}}$ must be
a rational function in $\lambda$. What differs here however is that $\Gamma(2)$
has three cusps (namely 0, 1, and $i\infty$). From our contour integration
argument in Proposition~\ref{prop-psiU} we see that we cannot have a pole at
the origin (in fact \eqref{eq:psi-0-exp} implies we must have a zero here), we
can (in fact must) have a pole at $i\infty$, and we may have unprescribed
behaviour at $\pm 1$.  This implies that the most we can conclude is that the
denominator of such a rational function, say $R(x)$, can only have factors of
the form $x^{a}(1-x)^{b}$ because $\lambda(0) =1$, $\lambda(1) = \infty$, and
$\lambda(i\infty) = 0$.

To continue, we will use  \eqref{eq-omega} to analyse the possible choices for
$R^{(k)}$. Combining \eqref{eq-omega} and \eqref{romega} yields
\begin{equation} \label{eq:funcr} 
R^{(k)}(\lambda(z)) = R^{(k)}(\lambda(Sz)) + R^{(k)}(\lambda(Tz)) 
\end{equation} 
We note that the field of meromorphic functions $\mathbb{C}(\lambda)$ is a
degree 6 field extension over the field of meromorphic functions
$\mathbb{C}(j)$ with the minimal polynomial of $\lambda$ over $\mathbb{C}(j)$
given by:
\begin{equation}
\label{eq-lambda}
\lambda^{6}-3\lambda^{5}+(6-j)\lambda^{4}-(7-2j)\lambda^{3}
+(6-j)\lambda^{2}-3\lambda+1 = 0
\end{equation}
Therefore, $R^{(k)}$ can be expressed in a unique way as
\begin{equation}
R^{(k)}(\lambda) = \sum^{5}_{m=0}R^{(k)}_{m}(j)\lambda^{m}
\end{equation}
for rational functions $R^{(k)}_{m}$. Inserting this into \eqref{eq:funcr} we
get
\begin{equation}
  \sum_{m=0}^{5}((1-\lambda)^{5}\lambda^{m}-(1-\lambda)^{5+m}+
  (-1)^{m}\lambda^{m}(1-\lambda)^{5-m})R^{(k)}_{m}(j) = 0
\end{equation}
We can use the minimal polynomial \eqref{eq-lambda} to write all powers of
$\lambda$ larger than 5 by linear combinations of
$\{1,\lambda,\ldots,\lambda^{5}\}$. This gives a linear system of 6 equations
for the 6 unknown functions $R_{m}^{(k)}$, $k =0,\ldots,5$. It can be checked
directly that this system has rank 4 and hence has a 2 dimensional kernel. This
supports an ansatz of the form
\begin{equation}
\label{eq-ansatz}
 \omega_{k}R^{(k)}(\lambda) = \chi_{1}^{(k)}Y^{(k)}(j) + \chi_{2}^{(k)}Z^{(k)}(j)
\end{equation} 
where the $Y^{(k)}$ and $Z^{(k)}$ are polynomials and $\chi_{1}^{(k)}$ and
$\chi_{2}^{(k)}$ are two linearly independent solutions of
\begin{equation}
  \chi(z) = z^{-2k}\chi(Sz) + \chi(Tz).  \label{eq:mod}
\end{equation}
Table~\ref{tab:chis} gives solutions of minimal orders at $z=0$ and $z=i\infty$.

\begin{table}[h]
  \setlength{\tabulinesep}{1mm}
  \centering
  \begin{tabu}[h]{|l|c|c|}
    \hline $k$&$\chi_1^{(k)}$&$\chi_2^{(k)}$\\\hline
    $0$&$\frac{(1+\lambda)(1-\lambda)(1-\lambda+\lambda^{2})}{\lambda^{2}}$
    &$\frac{(1+\lambda)(1-\lambda+\lambda^{2})}{\lambda(1-\lambda)}$\\\hline
    $1$&$\theta^{4}_{00}(1-\lambda)$
    &$\theta^{4}_{00}\frac{(1-\lambda)^{3}(2+3\lambda+2\lambda^{2})}{\lambda^{2}}$\\\hline
    $2$&$\theta_{00}^{8}(1-\lambda^{2})$
    &$\theta_{00}^{8}\frac{(1+\lambda)(1+3\lambda-7\lambda^{2}+3\lambda^{3}+\lambda^{4})}{\lambda(1-\lambda)}$\\\hline
    $3$&$\theta_{00}^{12}(1-\lambda)(1-\lambda+\lambda^{2})$
    &$\theta_{00}^{12}\frac{(1-\lambda+\lambda^{2})(1+3\lambda-10\lambda^{2}+3\lambda^{3}+\lambda^{4})}{\lambda(1-\lambda)}$\\\hline
    $4$&$\theta_{00}^{16}\lambda(1+\lambda)(1-\lambda)$
    &$\theta_{00}^{16}\frac{(1+\lambda)(1-\lambda +
      \lambda^{2}-\lambda^{3}+\lambda^{4}-\lambda^{5}+\lambda^{6})}{\lambda(1-\lambda)}$\\\hline
    $5$&$\theta_{00}^{20}\lambda(1-\lambda)(1-4\lambda+\lambda^{2})$
    &$\theta_{00}^{20}\frac{1-32\lambda^{3}+60\lambda^{4}-32\lambda^{5}+\lambda^{8}}{\lambda(1-\lambda)}$\\\hline
  \end{tabu}
  \medskip
  \caption{The choices for the forms $\chi_1^{(k)}$ and $\chi_2^{(k)}$}
  \label{tab:chis}
\end{table}
Putting all this information together we get that $\psi$ has the
form
\begin{equation} \label{psiform} \psi =
  \frac{1}{\Delta^{\ell}}\left(X^{(k)}(j)\omega_{k}\log\lambda +
    \chi_{1}^{(k)}Y^{(k)}(j)+\chi_{2}^{(k)}Z^{(k)}(j)\right)
\end{equation}
for polynomials $X^{(k)}, Y^{(k)}, Z^{(k)}$ that depend on the value of
$k$. Our next step will be to choose the degrees of $X^{(k)}, Y^{(k)}$, and
$Z^{(k)}$ and use the degrees of freedom given by the coefficients so that
\eqref{psiform} satisfies \eqref{eq:psi-0-exp}.  In particular this implies
that we need to choose their degrees so that $\psi$ vanishes to sufficiently
large order at $i\infty$. In particular, we want
\begin{equation} \label{eq:mod1} \begin{split}
\varphi(z) & : = z^{-2k}(X^{(k)}(j)\omega_{k}(Sz)\log\lambda(Sz) \\
& + \chi_{1}^{(k)}(Sz)Y^{(k)}(j(z))+\chi_{2}^{(k)}(Sz)Z^{(k)}(j(z))) =
\mathcal{O}(q^{\ell+\frac{1}{2}}). 
\end{split}
\end{equation} 
Before continuing in this direction however, we show two short lemmas.
\begin{lemma} \label{halflem} Suppose $\varphi(z)$ is as in \eqref{eq:mod1}.
  Then it has only half integer exponents in its Fourier expansion.
\end{lemma}
\begin{proof}
Let
\begin{equation*}
\chi(z)= \chi_{1}^{(k)}(z)Y^{(k)}(j(z))+\chi_{2}^{(k)}(z)Z^{(k)}(j(z))),
\end{equation*}
denote sum of the last two terms on the right side of \eqref{eq:mod1}.  Then
$\chi$ satisfies \eqref{eq:mod} and so
\begin{equation*}
z^{-2k}\chi(Sz)=    \chi(Tz)-\chi(z),
\end{equation*}
which implies that all terms in the Fourier expansion of $z^{-2k}\chi(Sz)$ with
integer exponents vanish. Moreover, we see from \eqref{logexps} that the
expression $z^{-2k}X^{(k)}(j)\omega_{k}(Sz)\log\lambda(Sz)$ has only half
integer exponents in its Fourier expansion, giving the claim.
\end{proof}
\begin{lemma} \label{evenorder} Let $\psi$ be given by \eqref{psiform} with
  polynomials $X$, $Y$, $Z$ such that \eqref{eq:psi-0-exp} holds. Then the
  principal part of $\psi$ at $i\infty$ has only integer exponents of $q$.
\end{lemma}
\begin{proof}
  By our assumption $z^{\frac d2-2}\psi(Sz)=\mathcal{O}(q^{\frac12})$. 
  Since
  $\psi$ can be written as
  \begin{equation*}
    \psi(z)=\sum_{k=-m}^\infty a_kq^{\frac k2}-iz\sum_{k=-n}^\infty b_kq^k
    = \psi_1+z\psi_2,
  \end{equation*}
  \eqref{eqn2} implies that $\psi_1$ satisfies 
  \begin{multline*}
    \psi(Tz)-2\psi(z)+\psi(T^{-1}z)=
    \psi_1(Tz)-2\psi_1(z)+\psi_1(T^{-1}z)\\
    =2\psi_1(Tz)-2\psi_1(z)=
    \mathcal{O}(q^{\frac12}),
  \end{multline*}
  which gives the assertion of the lemma.
\end{proof}

In light of Lemmas~\ref{halflem} and~\ref{evenorder}, we first assume that
\eqref{eq:psi-0-exp} holds and define the subscript $n$ for the polynomial
$X_{n}^{(k)}$ so that the following polar order is achieved.
 \begin{equation}
 \label{poleorderx}
 X^{(k)}_{n}(j)\omega_{k}  = \mathcal{O}(q^{-n}),
\end{equation} 
We note that this implies that for each $k \neq 1$ the degree of the polynomial
$X_{n}^{(k)}$ is at most $n$ and for $k =1$ that it has degree at most
$n-1$. We similarly adopt the notations $Y_{n}^{(k)}$ and $Z_{n}^{(k)}$ to
refer to the polynomials that give us:
\begin{align} \label{poleorderyz}
  \chi_{1}^{(k)}(z)Y^{(k)}_{n}(j(z))+\chi_{2}^{(k)}(z)Z^{(k)}_{n}(j(z)) & =
  \mathcal{O}(q^{-n-1})\\\label{poleordersy}
  z^{-2k}(\chi_{1}^{(k)}(Sz)Y^{(k)}_{n}(j(z)))& = \mathcal{O}(q^{-n+\frac{1}{2}})\\
\label{poleordersz}
z^{-2k}(\chi_{2}^{(k)}(Sz)Z^{(k)}_{n}(j(z))) &= \mathcal{O}(q^{-n+\frac{1}{2}}). 
\end{align}
We observe that \eqref{poleorderyz}, \eqref{poleordersy}, and
\eqref{poleordersz} are sufficient to put upper bounds on the degrees of
polynomials $Y_{n}^{(k)}$ and $Z_{n}^{(k)}$. With $b(k)$ as in
Table~\ref{tab:XYZ} we can use the degrees of freedom gained from the
coefficients of $X_{n}^{(k)}$, $Y_{n}^{(k)}$, and $Z_{n}^{(k)}$ so that
 \begin{equation} \label{optorder} \begin{split}
& z^{-2k}(X^{(k)}_{n}(j)\omega_{k}(Sz)\log\lambda(Sz) 
+ \chi_{1}^{(k)}(Sz)Y^{(k)}_{n}(j(z)) \\ & +\chi_{2}^{(k)}(Sz)Z^{(k)}_{n}(j(z)))
 = \mathcal{O}(q^{2n+\frac{b(k)}{2}}), 
\end{split}
\end{equation} 
which is a strengthening of our hypothesis that \eqref{eq:psi-0-exp} is
satisfied. We then observe that \eqref{poleorderx} and \eqref{poleorderyz}
ensure by \eqref{eq:W} that the largest real second order pole of $W(s)$ is 2
less than the largest real first order pole. Altogether, this will give us the
desired sign change of the function $F$ given by \eqref{Fdef}. The degrees of
these polynomials are also detailed in Table~\ref{tab:XYZ}.

 \begin{table}[h]
  \setlength{\tabulinesep}{1mm}
  \centering
  \begin{tabu}[h]{|l|c|c|c|c|}
    \hline
    $k$&$\deg X_n^{(k)}$&$\deg Y_n^{(k)}$
    &$\deg Z_n^{(k)}$&$b(k)$\\
    \hline
    $0$&$n$&$n$&$n-1$&$3$\\
    \hline
    $1$&$n-1$&$n$&$n$&$3$\\
    \hline
    $2$&$n$&$n+1$&$n$&$5$\\
    \hline
    $3$&$n$&$n+1$&$n$&$5$\\
    \hline
    $4$&$n$&$n+2$&$n+1$&$7$\\
    \hline
    $5$&$n$&$n+2$&$n+1$&$7$\\
    \hline
  \end{tabu}
  \medskip
   \caption{Degrees of the polynomials $X_n^{(k)}$,
    $Y_n^{(k)}$, and $Z_n^{(k)}$}\label{tab:XYZ}
\end{table}
We now need to choose $n$ so that
\begin{equation}
2n+\frac{b(k)}{2} > \ell 
\end{equation} 
so that \eqref{eq:psi-0-exp} is satisfied. This then gives that the minimal
choice of $n$ is then
\begin{equation} \label{eqmin}
  n=\left\lceil\frac{2\ell-b(k)}{4}\right\rceil.
\end{equation}
Then conditions \eqref{poleorderx} and \eqref{poleorderyz} ensure that there is
a sign change of $F(\mathbf{x})$ at $\|\mathbf{x}\|^2=2n+2\ell$+2 and
$F(\mathbf{x}) \neq 0$ for $\|\mathbf{x}\|^2=2n+2\ell$. Expressing $\ell, k,$
and $n$ in terms of $d$ yields $2n + 2\ell =2\lfloor\frac{d}{16}\rfloor$.

Summing up, we have proved most of the following theorem. The theorem is
formulated with some abuse of notation, which is justified by the fact that it
discusses radial functions: we write $F_-(\mathbf{x})=F_-(\|\mathbf{x}\|)$ and
consider $F_-$ as multivariate and univariate function as appropriate.

\begin{theorem}\label{thm-odd}
  For $d\equiv0\pmod4$ set $n_{-}=\lfloor\frac{d}{16}\rfloor + 1$. Then there
  exists a radial Schwartz function $F_{-}:\mathbb{R}^d\to\mathbb{R}$ satisfying
  \begin{equation}\label{eq:f-minus}
    \begin{split}
      F_{-}(\mathbf{x})&=(-1)^{\frac{d}{4}+1}\widehat{F}_{-}(\mathbf{x})
      \quad\text{for all }\mathbf{x}\in\mathbb{R}^d\\
   F_{-}(\sqrt{2n_{-}})&=0\quad\text{and }F'_{-}(\sqrt{2n_{-}})\neq0\\
   F_{-}(\sqrt{2m})&=F'_{-}(\sqrt{2m})=0\quad\text{for }m>n_-,
   \quad m\in\mathbb{N}.
    \end{split}
  \end{equation}
  \end{theorem} 
  \begin{proof}
    It remains to show that $F'_{-}(\sqrt{2n_{-}})\neq0$. This follows from the
    discussion of vanishing orders in Section~\ref{sec:diff-equat-gamma},
    especially equation \eqref{eq:phiw-ord} and the discussion following. More
    precisely, \eqref{eq:phiw-ord} implies that $a_{n_-}\neq0$ in
    \eqref{eq:psi-infty}, whereas \eqref{eq:phiwT-ord} implies that
    $b_{n_-}=0$. This shows by \eqref{eq:W} that $W$ has a simple pole at
    $s=2n_-$, which gives a simple zero of $U$ at this point.
  \end{proof}
  \begin{remark}\label{rem:extra-odd}
    Notice that by \eqref{optorder} the order increases by $2$, if $n$
    increases by $1$. Together with \eqref{eqmin} this gives that for certain
    $d\pmod{48}$ the order of vanishing of $\psi$ is one more than
    required. This means that for these values of the dimension there is one
    extra degree of freedom. This is the case for dimensions
  \begin{equation}
    \label{eq:extra-odd}
    d\equiv\{0,4,16,20,32,36\}\pmod{48}.
  \end{equation}
\end{remark}

\begin{table}
  \setlength{\tabulinesep}{1mm}
  \begin{tabu}[h]{|l|>{\raggedleft}p{6mm}|>{\raggedleft}p{6mm}|>{\raggedleft}p{5mm}|>{\raggedright}p{2.2cm}|>{\raggedright}p{2.2cm}|>{\raggedright}p{2.2cm}|}
    \cline{2-7}
    \multicolumn{1}{c}{}&\multicolumn{2}{|c|}{$d$}&$n$&\small$X_n^{(k)}(w)$&\small$Y_n^{(k)}(w)$&\small$Z_n^{(k)}(w)$\\
    \hline \multirow{2}{*}[2mm]{$k=0$} &
    $4$&$28$&$0$&\tiny$2$&\tiny$1$&\tiny$0$\\\cline{2-7}
    &$52$&$76$&$1$&\tiny$120\times (7
    w+384)$&\tiny$63 w+171776$&\tiny$91392$\\\hline
    \hline \multirow{2}{*}[2mm]{$k=1$} &
    &$24$&$0$&\tiny$0$&\tiny$0$&\tiny$1$\\\cline{2-7} &
    $48$&$72$&$1$&\tiny$840$&\tiny$514304-840 w$&\tiny$63 w+131584$\\\hline
    \hline \multirow{2}{*}[2mm]{$k=2$} &
    $20$&$44$&$0$&\tiny$6144$&\tiny$5 w+8192$&\tiny$-1280$ \\\cline{2-7}
    &$68$&$92$&$1$&\tiny$53760\times (3
    w+512)$&\tiny$33 w^2+202688 w+117014528$ &\tiny$-256\times (33 w+88256)$
    \\\hline \hline \multirow{2}{*}[2mm]{$k=3$} &
    $16$&$40$&$0$&\tiny$1536$&\tiny$5 w-9856$&\tiny$640$ \\\cline{2-7}
    &$64$&$88$&$1$&\tiny$215040\times (3 w+128)$
    &\tiny$231 w^2-26752 w-1267400704$ &\tiny$128\times (231 w+1002752)$
    \\\hline \hline \multirow{2}{*}[2mm]{$k=4$} &
    &$12$&$-1$&\tiny$0$&\tiny$w+768$&\tiny$-256$\\\cline{2-7}
    &$36$&$60$&$0$&\tiny$7864320$&\tiny$-7 w^2-14080
    w-3670016$&\tiny$256 (7 w+8704)$ \\\hline
    \hline \multirow{2}{*}[2mm]{$k=5$} &
    &$8$&$-1$&\tiny$0$&\tiny$w+1408$&\tiny$-256$\\\cline{2-7}
    &$32$&$56$&$0$&\tiny$55050240$&\tiny$-35 w^2-19456 w+89587712$
    &\tiny$256\times (35 w-29824)$
    \\\hline
  \end{tabu}
  \caption{The polynomials $X_n^{(k)}$, $Y_n^{(k)}$, and $Z_n^{(k)}$; the
    dimensions in the left column are those covered by
    Remark~\eqref{eq:extra-odd}.}
  \label{tab:XYZ1}
\end{table}


\section{Modular differential equations}
\label{sec:modular-diff-equat}
In Section~\ref{sec:determining-psi} we discussed the existence of the form
$\psi(Sz)z^{\frac d2-2}$ given by \eqref{eq:psi-quasi} in
Proposition~\ref{prop-even} by a simple dimension count. A similar argument was
used in Section~\ref{sec:determining-psi-odd} to obtain certain modular forms
for $\Gamma(2)$ that satisfy the requirements of Proposition~\ref{prop4.1}. In
the present section we give a direct construction by an explicit linear
recurrence, which we obtain via modular differential equations. For this
purpose we set
\begin{equation}
  \label{eq:f-def}
  f_{12n+2k+4}=\Delta^{n}\left(\omega_{k+2}P_n^{(k)}(j)-
    2\omega_{k+1}E_2Q_n^{(k)}(j)+\omega_{k}E_2^2R_n^{(k)}(j)\right),
\end{equation}
where $P_n^{(k)}$, $Q_n^{(k)}$, and $R_n^{(k)}$ are the polynomials used in the
ansatz \eqref{eq:psi-ansatz} studied in Section~\ref{sec:determining-psi}.
Similarly, we set
\begin{equation}
  \label{eq:phi-def}
  \phi_{12n+2k+12}=\Delta^{n+1}\left(\omega_kX_n^{(k)}\llambda+
    \chi_1^{(k)}Y_n^{(k)}(j)+
  \chi_2^{(k)}Z_n^{(k)}(j)\right),
\end{equation}
where $X_n^{(k)}$, $Y_n^{(k)}$, and $Z_n^{(k)}$ are the polynomials used in the
ansatz \eqref{psiform} in Section~\ref{sec:determining-psi-odd}.  In this
section we use the weight $w=12n+2k+4$ of the form $f_w$ and $w=12n+2k+12$ of
the form $\phi_w$ as the parameter.

\subsection{Differential equations for quasimodular forms}
\label{sec:diff-equat-quasi}
The quasimodular form $f_w$ of weight $w$ and depth $2$ is then given by the
three requirements
\begin{align}
  \label{eq:fw-ord}
  f_w&=\mathcal{O}(q^{3n+a(k)-1})=\mathcal{O}(q^{\lfloor\frac w4\rfloor-1})\\
  g_w&=\Delta^n\left(\omega_{k+1}Q_n^{(k)}(j)-\omega_{k}E_2R_n^{(k)}(j)\right)=
  \mathcal{O}(q)\label{eq:gw-ord}\\
  h_w&=\Delta^n\omega_kR_n^{(k)}(j)=\mathcal{O}(1)\label{eq:hw-ord}.
\end{align}
The third equation \eqref{eq:hw-ord} is of course trivially satisfied, we
mention it only, because we will show in the sequel that
\eqref{eq:fw-ord}--\eqref{eq:hw-ord} give the exact order for $z\to i\infty$ of
the functions $f_w$, $g_w$, and $h_w$ obtained by our construction and that the
solution is uniquely characterised by these conditions. We notice that forms
satisfying $f_w=\mathcal{O}(q^{\lfloor\frac w4\rfloor})$ are called
\emph{extremal} quasimodular forms; they were studied in
\cite{Kaneko_Koike2006:extremal_quasimodular_forms,
  Yamashita2010:construction_extremal_quasimodular}. We will adapt the methods
used in these two papers to our situation.

In \cite{Kaneko_Nagatomo_Sakai2017:third_order_modular} third order linear
differential equations are characterised, whose solution set is invariant under
the modular group. Especially, several cases of third order linear differential
equations are shown to have modular or quasimodular solutions. In
\cite{Grabner2020:quasimodular_forms} modular differential equations of orders
$\leq5$ having quasimodular solutions are fully characterised. Such
differential equations are especially useful for finding quasimodular forms
$f_w$ with prescribed behaviour of $f_w$, $g_w$ and $h_w$ for $q\to0$, because
the corresponding orders have to be solutions of the indicial equation. In this
section we will present this approach to finding the quasimodular forms
satisfying the requirements of Proposition~\ref{prop-even}.

Assuming the defining properties of the forms $f_w$ for even weights $w\geq8$,
there exist coefficients $a_w,b_w,c_w$ such that
\begin{equation}
  \label{eq:f_w-recurr}
  f_{w+4}=a_wE_4f_w+b_wE_4^2f_{w-4}+c_w\Delta f_{w-8}.
\end{equation}
This comes from considering the orders of the three forms on the right hand
side: the last two terms have the same orders for $q\to0$, thus the
coefficients $b_w$ and $c_w$ can be chosen so that the order of the sum
$b_wE_4^2f_{w-4}+c_w\Delta f_{w-8}$ equals the order of $E_4f_w$. Then $a_w$
can be chosen to again increase the order by $1$. Since the corresponding
functions $g_w$ and $h_w$ satisfy the same recurrence relation, these functions
automatically satisfy \eqref{eq:gw-ord} and \eqref{eq:hw-ord} by induction.  At
this moment this argument involves some heuristics, namely that there is no
higher order of vanishing in any term than expected. We will make this rigorous
by showing that the solutions of the linear recurrence satisfy certain
differential equations, which show that they have precisely the assumed order
of vanishing. This will then allow us to find values for $a_w$, $b_w$, and
$c_w$ and suitable initial values for $f_w$.

For the purpose of the proof we recall the definition of the Serre derivative
(see~\cite{Zagier2008:elliptic_modular_forms})
\begin{equation}
  \label{eq:serre}
  \partial_wf=f'-\frac w{12}E_2f.
\end{equation}
Notice that the Serre derivative $\partial_w$ applied to a modular form of
weight $w$ gives a modular form of weight $w+2$; similarly, the Serre
derivative $\partial_{w-2}$ applied to a quasimodular form of weight $w$ and
depth $2$ gives a quasimodular form of weight $w+2$ and depth $\leq2$; this
can be seen from
\begin{equation}
  \label{eq:serre-quasi}
  \begin{split}
    &\partial_{w-2}\left(A_w+E_2B_{w-2}+E_2^2C_{w-4}\right)=
    \partial_wA_w-\frac1{12}E_4B_{w-2}\\
    +&E_2\left(\frac16A_w+\partial_{w-2}B_{w-2}-\frac16E_4C_{w-4}\right)\\
    +&E_2^2\left(\frac1{12}B_{w-2}+\partial_{w-4}C_{w-4}\right),
  \end{split}
\end{equation}
where $A_w$, $B_{w-2}$, and $C_{w-4}$ are modular forms of respective weights
$w$, $w-2$, and $w-4$.

Furthermore, we recall the definition of the Rankin-Cohen bracket
(see~\cite{Zagier1994:modular_forms_differential})
\begin{equation}
  \label{eq:Rankin-Cohen}
  \left[f,g\right]_n^{(k,\ell)}=\sum_{i=0}^n(-1)^i\binom{n+k-1}{n-i}
  \binom{n+\ell-1}i f^{(i)}g^{(n-i)}.
\end{equation}
The cases $w\equiv0\pmod4$ (this corresponds to $k=0,2,4$) and $w\equiv2\pmod4$
(this corresponds to $k=1,3,5$) have to be treated slightly differently. In the
second case the underlying modular differential equation takes a somehow
non-standard form, which was not covered by the general study
\cite{Kaneko_Nagatomo_Sakai2017:third_order_modular}.
\goodbreak

\begin{proposition}\label{prop-equiv-0}
  Consider the sequence $(f_w)_w$, \textup{(}$w\equiv0\pmod4$,
  $w\geq8$\textup{)} given by the initial elements
\begin{equation}\label{eq:f-init}
\begin{split}
  f_8&=E_4^2-2E_2E_6+E_2^2E_4=\frac{36}5E_4''\\
  f_{12}&=\frac1{6000} (E_6^2 - 2 E_2 E_4 E_6 + E_2^2 E_4^2)=
  \frac1{3000}(E_4^2)''-\frac4{25}\Delta\\
  f_{16}&=\frac1{2540160000}\left(-25E_4^4 \!+ 49 E_4E_6^2 \!- 48 E_2 E_4^2 E_6
    \!+ E_2^2 (49 E_4^3 \!- 25 E_6^2)\right)\\
  &=\frac{(25E_6^2-49E_4^3)''}{2751840000}-\frac{\Delta E_4}{45500}  
\end{split}
\end{equation}
and the recurrence
\begin{equation}
  \label{eq:f0mod4-recurr}
  \begin{split}
    f_{w+4}&=\frac1{16000(w+2)(w-3)(w-5)(w-9)(w-10)(w-11)}\\
    &\times\Bigl(200(w-8)(w-9)(w^2-15w+38)E_4f_w\\&-
      \frac58(w-8)(w-12)E_4^2f_{w-4}+
      \Delta f_{w-8}\Bigr).
    \end{split}
\end{equation}
Then for every $w\equiv0\pmod4$, $w\geq8$ the quasimodular form $f_w$
satisfies \eqref{eq:fw-ord} and the corresponding functions $g_w$ and $h_w$
satisfy \eqref{eq:gw-ord} and \eqref{eq:hw-ord}, where $q^{\frac w4-1}$, $q$,
and $1$ are the exact orders of these functions for $q\to0$.
\end{proposition}
\begin{proof}
  We will prove the proposition by showing that $f_w$ satisfies the
  differential equation
\begin{equation}
  \label{eq:ODE1}
  \begin{split}
 f'''&-\frac{w}4E_2f''+\left(\frac{w-4}4E_4+\frac{w(w-1)}4E_2'\right)f'\\
    &-\left(\frac{(w-2)(w-4)}{16}E_4'+\frac{w(w-1)(w-2)}{24}E_2''\right)f=0.
  \end{split}
\end{equation}
This is the case $\alpha=\frac{w-4}4$ and $\beta=0$ of the general form of
linear differential equations admitting modular and quasimodular solutions
given in \cite[Theorem~1]{Kaneko_Nagatomo_Sakai2017:third_order_modular}. We
will write \eqref{eq:ODE1}\textsubscript{$w$} to indicate the dependence on the
parameter $w$ (so, $f_{w-4}$ is a solution of
\eqref{eq:ODE1}\textsubscript{$w-4$}). The indicial equation of \eqref{eq:ODE1}
then reads as
\begin{equation}
  \label{eq:index3}
  \mu^3-\frac{w}4\mu^2+\frac{w-4}4\mu=
  \mu\left(\mu-1\right)\left(\mu-\frac{w-4}4\right)=0.
\end{equation}
The roots of this equation correspond to the exponents in
\eqref{eq:fw-ord}--\eqref{eq:hw-ord}.

The differential equation \eqref{eq:ODE1} can be rewritten in terms of the
Serre derivative as
\begin{equation}
  \label{eq:ODE1Serre}
  \begin{split}
    \partial_{w+2}\partial_{w}\partial_{w-2}f&-
    \frac{3w^2-36w+140}{144}E_4\partial_{w-2}f\\
    &-\frac{(w-2)(w-5)(w-14)}{864}E_6f=0.
  \end{split}
\end{equation}

The following lemma can be checked by direct computation.
\begin{lemma}\label{lem-ode1}
  If $F_w$ is a solution of \eqref{eq:ODE1}\textsubscript{$w$}, then
the function
\begin{equation}
  \label{eq:soluw-4}
  \frac1{\Delta}\left(\left[F_w,E_4\right]_2^{(w-2,4)}+
    \frac53\left[F_w,E_6\right]_1^{(w-2,6)}\right)
\end{equation}
is a solution of \eqref{eq:ODE1}\textsubscript{$w-4$}.  
\end{lemma}

We now proceed by induction to show that the right-hand-side of
\eqref{eq:f0mod4-recurr} is a solution of \eqref{eq:ODE1}\textsubscript{$w+4$}.
Assume that we have proved that $f_m$ is a solution of
\eqref{eq:ODE1}\textsubscript{$m$} and that
\begin{equation}\label{eq:delta-f_m-4}
  \Delta f_{m-4}=\left[f_{m},E_4\right]_2^{(m-2,4)}+
    \frac53\left[f_{m},E_6\right]_1^{(m-2,6)}
\end{equation}
for $12\leq m\leq w$, $m\equiv0\pmod4$. Inserting \eqref{eq:delta-f_m-4} for
$m=w-4$ and $m=w$ into \eqref{eq:f0mod4-recurr} and using that $f_{w-4}$ and
$f_w$ are solutions of \eqref{eq:ODE1}\textsubscript{$w-4$} and
\eqref{eq:ODE1}\textsubscript{$w$} gives that the right-hand-side of
\eqref{eq:f0mod4-recurr} equals
\begin{equation}\label{eq:f-diff-recurr}
  \frac{(w-5)(w-6)E_4f_w-
    36\partial_w\partial_{w-2}f_w}{120(w+2)(w-3)(w-5)(w-10)}.
\end{equation}
This can be checked to be a solution of
\eqref{eq:ODE1}\textsubscript{$w+4$}. In order to complete the induction step,
we verify \eqref{eq:delta-f_m-4} for $m=w+4$ inserting the expression
\eqref{eq:f-diff-recurr}. This shows that indeed \eqref{eq:f0mod4-recurr} holds.

The solution of the recurrence \eqref{eq:f0mod4-recurr} is thus also a solution
of the differential equation \eqref{eq:ODE1}. Since $f_w$ vanishes to at least
second order in $q$ for $q\to0$, it can only be the solution of \eqref{eq:ODE1}
associated to the root $\mu=\frac{w-4}4$ of the indicial equation. Thus
\eqref{eq:fw-ord} holds. Similarly, since $g_w$ and $h_w$ satisfy the same
linear recurrence relation, we prove that \eqref{eq:gw-ord} and
\eqref{eq:hw-ord} hold.
\end{proof}


\begin{proposition}\label{prop-equiv-2}
  Consider the sequence $(f_w)_w$, \textup{(}$w\equiv2\pmod4$,
  $w\geq10$\textup{)} given by the initial elements
\begin{equation}\label{eq:f-init-2}
\begin{split}
  f_{10}&=-E_4E_6+2E_2E_4^2-E_2^2E_6=-\frac{24}7E_6''\\
  f_{14}&=-\frac1{8400} \left(E_4^2E_6-E_2(E_4^3+E_6^2)+E_2^2E_4E_6\right)\\
  &=
  -\frac{3}{19250}(E_4E_6)''-\frac{36}{875}E_2\Delta\\
  f_{18}&=-\frac1{237600}\left(5E_4^3E_6+7E_6^3-E_2(5E_4^4+19E_4E_6^2)
    +12E_2^2E_4^2E_6\right)\\
  &=-\frac1{28875}(E_4^2E_6)''+\frac{2 \Delta (181 E_6-185 E_2E_4)}{9625}
\end{split}
\end{equation}
and the recurrence
\begin{equation}
  \label{eq:f2mod4-recurr}
  \begin{split}
    f_{w+4}&=\frac1{16000(w-3)(w-4)(w-5)(w-9)(w-11)(w-16)}\\
    &\times\Bigl(200(w-9)(w-10)(w^2-21w+92)E_4f_w\\
    &-\frac58(w-10)(w-14)E_4^2f_{w-4}+\Delta f_{w-8}
    \Bigr).
    \end{split}
\end{equation}
Then for every $w\equiv2\pmod4$, $w\geq10$ the quasimodular form $f_w$
satisfies \eqref{eq:fw-ord} and the corresponding functions
$g_w$ and $h_w$ satisfy \eqref{eq:gw-ord} and \eqref{eq:hw-ord}, where
$q^{\frac{w-6}4}$, $q$, and $1$ are the exact orders of these functions for
$q\to0$.
\end{proposition}
\begin{proof}
    We will prove the proposition by showing that $f_w$ satisfies the
  differential equation
\begin{equation}
  \label{eq:ODE2}
  \begin{split}
 &E_6f'''-\left(\frac{w-2}4E_4^2+\frac w2E_6'\right)f''\\
    +&\left(\frac{w-6}4E_4E_6+\frac{(w-1)(w-2)}8E_4E_4'+
      \frac{w(w-1)}{14}E_6''\right)f'\\
    -&\left(\frac{(w-2)(w-6)}{24}E_4E_6'+\frac{5(w-8)(w-9)(w-10)}{384}(E_4')^2
    \right.\\
    &\left.+\frac{w^3+105w^2-1162w+3576}{480}E_4E_4''+
      \frac{w(w-1)(w-2)}{336}E_6'''\right)f=0.
  \end{split}
\end{equation}
The extra factor $E_6$ in front of the highest derivative is motivated by the
computations in \cite{Yamashita2010:construction_extremal_quasimodular}. The
indicial equation is then given by
\begin{equation*}
  \mu^3-\frac{w-2}4\mu^2+\frac{w-6}4\mu=
  \mu(\mu-1)\left(\mu-\frac{w-6}4\right)=0.
\end{equation*}
The equation \eqref{eq:ODE2} can be written in terms of Serre derivatives as
\begin{equation}
  \label{eq:ODE2-Serre}
  \begin{split}
    &E_6\partial_{w+2}\partial_w\partial_{w-2}f+
    \frac12E_4^2\partial_w\partial_{w-2}f-
    \frac{3w^2-48w+224}{144}E_4E_6\partial_{w-2}f\\
    &- \left(\frac{3w^2-48w+224}{288}E_4^3+
      \frac{w^3-33w^2+300w-896}{864}E_6^2\right)f=0
  \end{split}
\end{equation}
Then a relation similar to Lemma~\ref{lem-ode1} holds.
\begin{lemma}\label{lem-ode2}
  If $F_w$ is a solution of \eqref{eq:ODE2}\textsubscript{$w$}, then
the function
\begin{equation}
  \label{eq:soluw-4-1}
  \frac1{\Delta}\left(\left[F_w,E_4\right]_2^{(w-2,4)}+
    \frac53\left[F_w,E_6\right]_1^{(w-2,6)}\right)
\end{equation}
is a solution of \eqref{eq:ODE2}\textsubscript{$w-4$}.  
\end{lemma}

We proceed in a similar manner as in the proof of
Proposition~\ref{prop-equiv-0}. The main equation after inserting
\eqref{eq:soluw-4-1} into the recurrence and using that $f_w$ is a solution of
\eqref{eq:ODE2} reads
\begin{equation}
  \label{eq:f-diff-recurr-2}
  f_{w+4}=\frac{(w-8)(w-9)E_4f_w-36\partial_w\partial_{w-2}f_w}
  {120(w-3)(w-4)(w-5)(w-16)}
\end{equation}
This can be checked again to satisfy \eqref{eq:ODE2}\textsubscript{w+4} and
\eqref{eq:delta-f_m-4}\textsubscript{$w+4$}. The remaining arguments are also
similar.
\end{proof}
\subsection{Differential equations for $\Gamma(2)$-modular
  forms}\label{sec:diff-equat-gamma}
The forms $\phi_w$ are given by the requirements
\begin{align}
  \label{eq:phiw-ord}
  \phi_w(z)&=\mathcal{O}(1)\\
  \phi_w(z)-\phi_w(Tz)&=\mathcal{O}(q)\label{eq:phiwT-ord}\\
  z^{-w}\phi_w(Sz)&=\mathcal{O}(q^{\lfloor\frac w4\rfloor-\frac12}).
  \label{eq:phiwS-ord}
\end{align}
Notice that \eqref{eq:phiwT-ord} is equivalent to $f\Delta^{n+1}$, with $f$ as
in \eqref{eq:psi-sol}, being a cusp form. Furthermore, we will show that the
orders of vanishing for $z\to i\infty$ given in
\eqref{eq:phiw-ord}--\eqref{eq:phiwS-ord} are exact. Similarly to
\eqref{eq:f_w-recurr} we expect a linear recurrence equation for $\phi_w$ by
the same heuristic argument. The fact that such recurrences indeed hold, is the
content of the following two propositions. Their proofs follow exactly the same
lines as the proofs of Propositions~\ref{prop-equiv-0} and~\ref{prop-equiv-2}.

\begin{proposition}\label{prop-phi-equiv-0}
  Consider the sequence $(\phi_w)_w$, \textup{(}$w\equiv0\pmod4$,
  $w\geq8$\textup{)} given by the initial elements
\begin{equation}\label{eq:phi-init}
\begin{split}
  \phi_8&=\theta_{01}^{12} \left(\theta_{01}^4+2 \theta_{10}^4\right)\\
  \phi_{12}&=\frac8{175} \Delta \llambda
  +\frac{\theta_{01}^{12}}{11200}\left(2 \theta_{10}^{12}+
    3 \theta_{01}^{4} \theta_{10}^{8}+
    3 \theta_{01}^{8} \theta_{10}^4+\theta_{01}^{12}
  \right)\\
  \phi_{16}&=\frac1{231000} \Delta E_4 \llambda\\&+
  \frac{\theta_{01}^{12}}{1419264000}
  \left(24\theta_{10}^{20}+60 \theta_{01}^{4} \theta_{10}^{16}
    +68 \theta_{01}^{8} \theta_{10}^{12}+42 \theta_{01}^{12} \theta_{10}^8\right)
\end{split}
\end{equation}
and the recurrence (for $w\geq16$)
\begin{equation}
  \label{eq:phi0mod4-recurr}
  \begin{split}
    \phi_{w+4}&=\frac1{16000(w+4)(w-1)(w-3)(w-7)(w-8)(w-9)}\\
    &\times\Bigl(200(w-6)(w-7)(w^2-11w+12)E_4\phi_w\\&-
      \frac58(w-6)(w-10)E_4^2\phi_{w-4}+
      \Delta \phi_{w-8}\Bigr).
    \end{split}
\end{equation}
Then for every $w\equiv0\pmod4$, $w\geq8$ the modular form $\phi_w$ satisfies
\eqref{eq:phiw-ord}, \eqref{eq:phiwT-ord}, and~\eqref{eq:phiwS-ord}, where $1$,
$q$, and $q^{\frac {w-2}4}$ are the exact orders of these functions for
$q\to0$.
\end{proposition}
\begin{proof}
  As pointed out, the proof follows the same lines as  the proof of
  Proposition~\ref{prop-equiv-0}; it only has to be shown that $\phi_w$
  satisfies \eqref{eq:ODE1}\textsubscript{$w+2$}. This is done by showing the
  relation
  \begin{equation}
    \label{eq:phi-diff-recurr-0}
    \phi_{w+4}=\frac{(w-3)(w-4)E_4\phi_w-36\partial_{w+2}\partial_w\phi_w}
    {120(w+4)(w-1)(w-3)(w-8)}.
  \end{equation}
\end{proof}
\begin{proposition}\label{prop-phi-equiv2}
  Consider the sequence $(\phi_w)_w$, \textup{(}$w\equiv2\pmod4$,
  $w\geq10$\textup{)} given by the initial elements
\begin{equation}\label{eq:phi-init-2}
  \begin{split}
    \phi_{10}&=\theta_{01}^{12} \left(5 \theta_{10}^8+
      5 \theta_{01}^4 \theta_{10}^4+
      2 \theta_{01}^8\right)\\
    \phi_{14}&=\frac{\theta_{01}^{20}}{13440} \left(
      7 \theta_{10}^8+ 7 \theta_{01}^4 \theta_{10}^4+
    2 \theta_{01}^8\right)\\
  \phi_{18}&=\frac1{600600} \Delta E_6 \llambda\\&+
  \frac{\theta_{01}^{12}}{1845043200}\Bigl(-12  \theta_{10}^{24}
    -36 \theta_{01}^{4} \theta_{10}^{20}-  13 \theta_{01}^{8} \theta_{10}^{16}
    +34 \theta_{01}^{12} \theta_{10}^{12}\\
    &\qquad\qquad\qquad\quad+68 \theta_{01}^{16} \theta_{10}^8
    +45 \theta_{01}^{20} \theta_{10}^4+10 \theta_{01}^{24}\Bigr)\\
\end{split}
\end{equation}
and the recurrence (for $w\geq18$)
\begin{equation}
  \label{eq:phi2mod4-recurr}
  \begin{split}
    \phi_{w+4}&=\frac1{16000(w-1)(w-2)(w-3)(w-7)(w-9)(w-14)}\\
    &\times\Bigl(200(w-7)(w-8)(w^2-17w+54)E_4\phi_w\\
    &-\frac58(w-8)(w-12)E_4^2\phi_{w-4}+\Delta \phi_{w-8}
    \Bigr).
    \end{split}
\end{equation}
Then for every $w\equiv2\pmod4$, $w\geq10$ the modular form $\phi_w$ satisfies
\eqref{eq:phiw-ord}, \eqref{eq:phiwT-ord}, and~\eqref{eq:phiwS-ord}, where $1$,
$q$, and $q^{\frac{w-4}4}$ are the exact orders of these functions for $q\to0$.
\end{proposition}
\begin{proof}
  The proof is again similar to the proof of Proposition~\ref{prop-equiv-2}; it
  only has to be shown that $\phi_w$ satisfies
  \eqref{eq:ODE2}\textsubscript{$w+2$}. This is done by showing the relation
  \begin{equation}
    \label{eq:phi-diff-recurr-2}
    \phi_{w+4}=\frac{(w-6)(w-9)E_4\phi_w-36\partial_{w+2}\partial_w\phi_w}
    {120(w-1)(w-2)(w-3)(w-14)}.
  \end{equation}
\end{proof}
\section{Positivity of the coefficients}\label{sec:posit-coeff}

In this section we will show that the forms $f_w$ obtained in
Propositions~\ref{prop-equiv-0} and~\ref{prop-equiv-2} can be written in the
form
\begin{equation}
  \label{eq:fw-decomp}
  f_w=(-1)^{w/2}\frac{144\mu_w}{(w-3)(w-4)}E_{w-4}''+\Delta \tilde{f}_{w-12},
\end{equation}
where $\tilde{f}_{w-12}$ is a quasimodular form of weight $w-12$ and depth
$2$. The factor $\mu_w$ is given by
\begin{equation}
  \label{eq:mu-def}
  \mu_w=
  \begin{cases}
    \frac{3(w/4-2)!}{80^{w/4-2}(w-7)!!(w/2-1)!!}&\text{for }w\equiv0\pmod4\\
   \frac{3((w-10)/4)!}{80^{(w-10)/4}(w-7)!!(w/2-4)!!}&\text{for }w\equiv2\pmod4,
  \end{cases}
\end{equation}
where $m!!=m(m-2)(m-4)\ldots$ denotes the double factorial.
From this we derive the following proposition.
\begin{proposition}\label{prop-pos}
  Let $f_w$ be the quasimodular forms given by Propositions~\ref{prop-equiv-0}
  and~\ref{prop-equiv-2}. Then all Fourier coefficients of $f_w$ are positive
  with possibly finitely many exceptions.
\end{proposition}

\begin{proof}
  We set
  \begin{equation}\label{eq:f_w-ABC}
    f_w=A_w+E_2B_{w-2}+E_2^2C_{w-4}.
  \end{equation}
  Applying \eqref{eq:serre-quasi} twice we obtain
  \begin{equation}
    \label{eq:serre-quasi2}
    \begin{split}
      &\partial_w\partial_{w-2}\left(A_w+E_2B_{w-2}+E_2^2C_{w-4}\right)=\\
      &\frac1{72}\left(72\partial_{w+2}\partial_wA_w-E_4A_w-
        12E_4\partial_{w-2}B_{w-2}
      +2E_6B_{w-2}+E_4^2C_{w-4}\right)\\+
      &\frac{E_2}{36}\left(12\partial_wA_w+\partial_w\partial_{w-2}B_{w-2}
        -E_4B_{w-2}-12E_4\partial_{w-4}C_{w-4}+
        2E_6C_{w-4}\right)\\+
      &\frac{E_2^2}{72}\left(A_w+12\partial_{w-2}B_{w-2}+
        72\partial_{w-2}\partial_{w-4}C_{w-4} -E_4C_{w-4}\right).
    \end{split}
  \end{equation}
Setting $a_w$, $b_{w-2}$, and $c_{w-4}$ for the constant coefficients of $A_w$,
$B_{w-2}$, and $C_{w-4}$ and using \eqref{eq:f-diff-recurr} we
obtain \tiny
\begin{equation*}
  \left(\begin{array}{l}
    a_{w+4}\\b_{w+2}\\c_w
  \end{array}\right)
  =\frac{\left(
\begin{array}{ccc}
 3 w^2-46 w+122 & -2 w & -2 \\
 4 w & 3 w^2-42 w+124 & 8-4 w \\
 -2 & 2 (w-2) & 3 w^2-38 w+114 \\
\end{array}
\right)
 \left(\begin{array}{l}
    a_{w}\\b_{w-2}\\c_{w-4}
  \end{array}\right)}{480 (w - 10) (w - 5) (w - 3) (w + 2)},
\end{equation*}
\normalsize
from which we obtain
\begin{equation*}
  \left(\begin{array}{l}
    a_{w}\\b_{w-2}\\c_{w-4}
  \end{array}\right)=\mu_w
\left(\begin{array}{r}
  1\\-2\\1
\end{array}\right)
\end{equation*}
for $w\equiv0\pmod4$; using \eqref{eq:f-diff-recurr-2} for $w\equiv2\pmod4$
gives a similar recursion. This shows that all
three forms $A_w$, $B_{w-2}$, and $C_{w-4}$ have non-vanishing Eisenstein
parts.

The form $f_w$ can be rewritten as
\begin{equation}\label{eq:fw-deriv}
  f_w=\alpha_{w-4}''+\beta_{w-2}'+\gamma_w
\end{equation}
with forms $\alpha_{w-4}$, $\beta_{w-2}$, and $\gamma_w$ of respective weights
$w-4$, $w-2$, and $w$. The two representations \eqref{eq:f_w-ABC} and
\eqref{eq:fw-deriv} are related by
\begin{equation}
  \label{eq:fw-deriv-ABC}
  \begin{split}
    \alpha_{w-4}&=\frac{144}{(w-3)(w-4)}C_{w-4}\\
    \beta_{w-2}&=\frac{12}{w-2}\left(B_{w-2}-
      \frac{24}{w-4}\partial_{w-4}C_{w-4}\right)\\
    \gamma_w&=A_w-\frac{12}{w-2}\partial_{w-2}B_{w-2}+\frac1{w-3}E_4C_{w-4}\\
    &+\frac{144}{(w-2)(w-3)}\partial_{w-2}\partial_{w-4}C_{w-4}.
  \end{split}
\end{equation}
Inserting our previous result that the constant coefficients of
$A_w$,$B_{w-2}$, and $C_{w-4}$ are proportional to $(1,-2,1)$ yields that
$\beta_{w-2}$ and $\gamma_w$ are cusp forms, thus multiples of $\Delta$. This
yields the decomposition \eqref{eq:fw-decomp} and more precisely
\begin{equation}\label{eq:f_w-eisenstein}
  f_w=(-1)^{w/2}\frac{144\mu_w}{(w-3)(w-4)}E_{w-4}''+\widetilde{\alpha}_{w-4}''
  +\beta_{w-2}'+\gamma_w
\end{equation}
with cusp forms $\widetilde{\alpha}_{w-4}$, $\beta_{w-2}$, and $\gamma_w$.

The Fourier coefficients of $E_{w-4}''$ are
$-\frac{2(w-4)}{\mathcal{B}_{w-4}}n^2\sigma_{w-5}(n)$ and thus of order
$n^{w-3}$ (here $\mathcal{B}_{w-4}$ denotes the Bernoulli numbers). Using
Deligne's estimate (see~\cite{Deligne1974:la_conjecture_weil}) for the
coefficients of cusp forms to estimate the coefficients of
$\widetilde{\alpha}_{w-4}''+\beta_{w-2}'+\gamma_w$ gives an estimate of order
$n^{\frac{w-1}2}\sigma_0(n)$ ($\sigma_0(n)$ being the number of positive
divisors of $n$) for the coefficients of $\Delta\tilde{f}_{w-12}$ in
\eqref{eq:fw-decomp}. Thus the coefficient of the second term in
\eqref{eq:fw-decomp} is of smaller order than the coefficient of the first
term, which then determines the sign of all but possibly finitely many Fourier
coefficients.
\end{proof}
We explain shortly, how to prove positivity of all coefficients for a fixed
value of $w$: we start from \eqref{eq:f_w-eisenstein}. Then the
coefficients of the cusp forms can be estimated by an explicit bound obtained
in \cite{Jenkins_Rouse2011:bounds_coefficients_cusp}.
\begin{theorem}[Theorem~1 in \cite{Jenkins_Rouse2011:bounds_coefficients_cusp}]
  Let
  \begin{equation*}
    G(z)=\sum_{n=1}^\infty g(n)q^n
  \end{equation*}
  be a cusp form of weight $w$. Then
  \begin{equation}
    \label{eq:jenkins-rouse}
    \begin{split}
      |g(n)|&\leq\sqrt{\log w}\Biggl(
        11\sqrt{\sum_{m=1}^\ell\frac{|g(m)|^2}{m^{w-1}}}\\
      &+
        \frac{e^{18.72}(41.41)^{w/2}}{w^{(w-1)/2}} \left|\sum_{m=1}^\ell
          g(m)e^{-7.288m}\right|\Biggr)n^{\frac{w-1}2}\sigma_0(n),
    \end{split}
\end{equation}
where $\ell$ is the dimension of the space of cusp forms of weight $w$.
\end{theorem}
Using this bound we obtain that the coefficients of
$\widetilde{\alpha}_{w-4}''$, $\beta_{w-2}'$, and $\gamma_w$ are bounded by
\begin{equation*}
  C_in^{\frac{w-1}2}\sigma_0(n)\quad\text{for }i=\alpha,\beta,\gamma
\end{equation*}
for explicit constants $C_\alpha$, $C_\beta$, and $C_\gamma$. Then the
coefficients of $f_w$ are positive for all $n$ satisfying
\begin{equation}
  \label{eq:fw-fourier-bound}
  \frac{144\mu_w}{(w-3)(w-4)}\frac{2(w-4)}{|\mathcal{B}_{w-4}|}
  n^2\sigma_{w-5}(n)>(C_\alpha+C_\beta+C_\gamma)n^{\frac{w-1}2}\sigma_0(n).
\end{equation}
Estimating $\sigma_{w-5}(n)\geq n^{w-5}$ and $\sigma_0(n)<2\sqrt n$, this can
be solved explicitly for the minimal value of $n$; the remaining finitely many
values of $n$ can be checked with the help of a computer. The number of
coefficients to be checked was up to $3300$ in the cases we studied. For the
values $w\leq22$ the cusp forms are either trivial or Hecke eigenforms, because
the space of cusp forms has dimension $\leq1$. In this case Deligne's estimate
can be used directly, and the number of coefficients to be checked is less than
$10$.

\begin{remark}\label{rem1}
  We have checked the positivity of the Fourier coefficients of the forms $f_w$
  for even $w$ in the range $8,\ldots,94$, which corresponds to dimensions
  $d=4,\ldots,312$ ($d$ divisible by $4$). Notice that the weight depends on
  dimension by the relation
  \begin{equation*}
    w=12\left\lfloor\frac{d+4}{16}\right\rfloor-\frac d2+16.
  \end{equation*}
  As pointed out in Theorem~\ref{thm-even} this implies that
  $F_+(\mathbf{x})\leq0$ for $\|\mathbf{x}\|^2\geq 2n+2\ell-2$ for these
  dimensions.
\end{remark}

The numerical experiments support the following conjecture.
\begin{conjecture}\label{conj1}
  Let $f_w$ ($w\equiv0\pmod2$) be the quasimodular form given by
  Propositions~\ref{prop-equiv-0} and~\ref{prop-equiv-2}. Then
  \begin{equation*}
    f_w=\sum_{m=\lfloor\frac w4\rfloor-1}^\infty a_w(m)q^m
  \end{equation*}
  with $a_w(m)>0$ for $m\geq \lfloor\frac w4\rfloor-1$.
\end{conjecture}
This conjecture is similar to Conjecture~2 in
\cite{Kaneko_Koike2006:extremal_quasimodular_forms}, which states the
positivity of the Fourier coefficients of extremal quasimodular forms of
depth $\leq4$.

\begin{remark}
  The Fourier coefficients of the modular forms $\phi_w$ do not seem to be
  positive for small values of $w$. Numerical experiments indicate that the
  functions $\phi_w(it)$ are positive for $t>0$. Except for $w=8,10,14$
  (corresponding to dimensions $8$, $12$, and $24$) this
  seems to be difficult to prove due to the presence of the $\llambda$-term.
\end{remark}

\section{Examples: small dimensions}
\label{sec:concluding-remarks}
For some small dimensions the functions we constructed in
Sections~\ref{sec:determining-psi} and~\ref{sec:determining-psi-odd} are of
special interest. These of course include the dimensions $8$, $12$, and $24$
studied in \cite{Cohn-Kumar-Miller+2017:sphere_packing_24,
  Viazovska2017:sphere_packing_8,
  Cohn_Goncalves2019:optimal_uncertainty_principle}. In the following we
express all $\Gamma(2)$-modular functions in terms of $\theta$-functions by
replacing $\lambda$ using \eqref{eq:lambda-theta}.

\subsection*{Dimension $4$}
The function $\lambda$ maps the positive imaginary axis to the interval
$(0,1)$. Thus the function $\llambda$ produces a non-positive $+1$
eigenfunction for the Fourier transform. The function $\chi_1^{(0)}$ turns into
a $+1$ eigenfunction vanishing at $\mathbf{0}$ and having a last sign change at
distance $\sqrt2$. This shows $A_+(4)\leq\sqrt2$

The function
\begin{equation*}
  8400\frac{f_{14}}\Delta=\frac{E_4^2E_6-E_2(E_4^3+E_6^2)+E_2^2E_4E_6}\Delta
\end{equation*}
produces a Fourier eigenfunction for the eigenvalue $-1$, which does not vanish
at $\mathbf{0}$ and has a last sign change at distance $\sqrt2$.  This shows
$A_-(4)\leq\sqrt2$.
\subsection*{Dimension $8$}
Our results recover the functions used in \cite{Viazovska2017:sphere_packing_8}
to prove the optimal upper bound for sphere packings in dimension $8$: the
quasimodular form used there is
\begin{equation*}
  6000\frac{f_{12}}{\Delta}=\frac{E_6^2-2E_2E_4E_6+E_2^2E_4^2}{\Delta}.
\end{equation*}
The second modular form giving the $-1$ Fourier eigenfunction is
\begin{equation*}
  \frac{\phi_{10}}\Delta=\frac{\theta_{01}^{12}\left(5 \theta_{10}^8+
      5 \theta_{01}^4 \theta_{10}^4+
      2 \theta_{01}^8\right)}\Delta.
\end{equation*}

\subsection*{Dimension $12$}
Similarly, the function used in
\cite{Cohn_Goncalves2019:optimal_uncertainty_principle} to show an optimal
uncertainty principle in dimension $12$ is given by
\begin{equation*}
  \frac{\phi_8}\Delta=\frac{\theta_{01}^{12} \left(\theta_{01}^4+2
      \theta_{10}^4\right)}\Delta.
\end{equation*}

The function
\begin{multline*}
 C_{22}\frac{f_{22}}{\Delta^2}=\frac1{\Delta^2}\Bigl(205E_4^4E_6-637E_4E_6^3
 +E_2(70E_4^5+ 794E_4^2E_6^2)\\
 +E_2^2(275E_6^3-707E_4^3E_6)\Bigr)
\end{multline*}
($C_{22}$ being some large integer constant) provides a $-1$ eigenfunction,
which does not vanish at $\mathbf{0}$ and has a last sign change at
$2$. In this case Remark~\ref{rem:extra-even} applies and we can achieve a
more suitable function by forming a linear combination of $f_{22}$ and
$E_4f_{18}$ to obtain the function
\begin{multline*}
  \frac1{\Delta^2}\Bigl(415 E_4^4 E_6+161 E_4 E_6^3)
  -2 E_2 \left(431 E_4^2 E_6^2+145 E_4^5\right)\\
  +E_2^2 \left(451 E_4^3 E_6+125 E_6^3\right)\Bigr),
\end{multline*}
which transforms into an eigenfunction vanishing at $\mathbf{0}$ and having a
last sign change at distance $2$. This shows $A_-(12)\leq2$.
\subsection*{Dimension $16$}
This is the smallest dimension for which both Remarks~\ref{rem:extra-even}
and~\ref{rem:extra-odd} apply.

The function
\begin{multline*}
  C_{20}\frac{f_{20}}{\Delta^2}=
  \frac1{\Delta^2}\Bigl(-469 E_4^2 E_6^2+325 E_4^5+
  E_2 \left(358 E_4^3 E_6-70 E_6^3\right)\\
  +E_2^2 \left(395 E_4 E_6^2-539 E_4^4\right)\Bigr)
\end{multline*}
transforms into a Fourier eigenfunction for the eigenvalue $1$, which has a
last sign change at distance $2$. By Remark~\ref{rem:extra-even} we have an
extra degree of freedom by forming a linear combination of $f_{20}$ and
$E_4f_{16}$ to obtain the function
\begin{multline*}
  \frac1{\Delta^2}\Bigl(203 E_4^2 E_6^2+85 E_4^5-2 E_2 \left(211 E_4^3 E_6+77
    E_6^3\right)\\
  +  E_2^2 \left(239 E_4 E_6^2+49 E_4^4\right)\Bigr),
\end{multline*}
which transforms into a Fourier eigenfunction vanishing at $\mathbf{0}$ and
having a last sign change at distance $2$. This shows $A_+(16)\leq2$.

The functions
\begin{multline*}
  C_{18}\frac{\phi_{18}}{\Delta^2}=\frac{3072E_6}{\Delta}\llambda+
  \frac{\theta_{01}^{12}}{\Delta^2}\Bigl(-12  \theta_{10}^{24}
    -36 \theta_{01}^{4} \theta_{10}^{20}-  13 \theta_{01}^{8} \theta_{10}^{16}\\
    +34 \theta_{01}^{12} \theta_{10}^{12}+68 \theta_{01}^{16} \theta_{10}^8
    +45 \theta_{01}^{20} \theta_{10}^4+10 \theta_{01}^{24}\Bigr)
  \end{multline*}
  and
  \begin{equation*}
    C_{14}\frac{E_4\phi_{14}}{\Delta^2}=
    \frac{\theta_{01}^4
      \left(\theta_{01}^8+\theta_{01}^4 \theta_{10}^4+\theta_{10}^8\right)
      \left(7\theta_{10}^8+7 \theta_{01}^4 \theta_{10}^4+2
        \theta_{01}^8\right)}{\theta_{10}^{16} \theta_{00}^{16}}
  \end{equation*}
  both transform into Fourier eigenfunctions with last sign change at distance
  $2$. The transform of the second function vanishes at $\mathbf{0}$. This
  shows $A_-(16)\leq2$.
\subsection*{Dimension $20$}
The function
\begin{equation*}
  C_{18}\frac{f_{18}}{\Delta^2}=
  \frac1{\Delta^2}\Bigl(5 E_4^3 E_6+7 E_6^3
  +E_2 \left(-19 E_4 E_6^2-5 E_4^4\right)+
  12 E_2^2 E_4^2 E_6\Bigr)
\end{equation*}
transforms into a Fourier eigenfunction for eigenvalue $-1$, which has its last
sign change at distance $2$. This eigenfunction does not vanish at
$\mathbf{0}$. This shows $A_-(20)\leq2$.

The functions
\begin{multline*}
  C_{16}\frac{\phi_{16}}{\Delta^2}=3\frac{E_4}\Delta\llambda\\
  +  \frac{32 \left(5 \theta_{01}^{20}+20 \theta_{01}^{16} \theta_{10}^{4}
      +42 \theta_{01}^{12} \theta_{10}^8+68 \theta_{01}^8 \theta_{10}^{12}
      +60 \theta_{01}^4 \theta_{10}^{16}+24
   \theta_{10}^{20}\right)}{\theta_{01}^4 \theta_{10}^{16} \theta_{00}^{16}}
\end{multline*}
and
\begin{equation*}
  C_{12}\frac{E_4\phi_{12}}{\Delta^2}=\frac{E_4}\Delta\llambda+
  \frac{128 (\theta_{01}^4+2 \theta_{10}^4) \left(\theta_{01}^8
      +\theta_{01}^4 \theta_{10}^4+\theta_{10}^8\right)^2}
  {\theta_{01}^4 \theta_{10}^{16} \theta_{00}^{16}}
\end{equation*}
both transform into Fourier eigenfunctions with a last sign change at distance
$2$. Forming the linear combination
\begin{equation*}
  \frac{\theta_{01}^4 \left(\theta_{01}^{12}+4 \theta_{01}^8 \theta_{10}^4
      +6 \theta_{01}^4 \theta_{10}^8+4 \theta_{10}^{12}\right)}
  {\theta_{10}^{16}\theta_{00}^{16}}
\end{equation*}
gives a function, which transforms into a Fourier eigenfunction vanishing at
$\mathbf{0}$ and having sign changes exactly at distances $\sqrt{2}$ and
$2$. This shows $A_+(20)\leq2$.
\subsection*{Dimension $24$}
The functions used in \cite{Cohn-Kumar-Miller+2017:sphere_packing_24} to obtain
the optimal upper bound for sphere packings in dimension $24$ were
\begin{equation*}
  2540160000\frac{f_{16}}{\Delta^2}=
  \frac{-25E_4^4 \!+ 49 E_4E_6^2 \!- 48 E_2 E_4^2 E_6
      \!+ E_2^2 (49 E_4^3 \!- 25 E_6^2)}{\Delta^2}
\end{equation*}
and
\begin{equation*}
  13440\frac{\phi_{14}}{\Delta^2}=\frac{\theta_{01}^{20}\left(
      7 \theta_{10}^8+ 7 \theta_{01}^4 \theta_{10}^4+
    2 \theta_{01}^8\right)}{\Delta^2}.
\end{equation*}

\subsection*{Dimension $48$}
The function
\begin{equation*}
  \begin{split}
    C_{28}\frac{f_{28}}{\Delta^4}&=\frac1{\Delta^4} \Bigl(-9565298 E_4^4
    E_6^2+1101373 E_4 E_6^4+7254325 E_4^7\\
    &+ 48 E_2 \left(133387 E_4^5 E_6-82987
      E_4^2 E_6^3\right)\\
    &+ E_2^2 \left(10650578 E_4^3 E_6^2-11603053
      E_4^6-257125 E_6^4\right)\Bigr)
  \end{split}
\end{equation*}
  ($C_{28}$ being some large integer constant) transforms into a Fourier
eigenfunction for the eigenvalue $1$, which has a last sign change at distance
$\sqrt8$. In this case
Remark~\ref{rem:extra-even} applies and we can achieve a more suitable function
by forming a linear combination of $f_{28}$ and $E_4f_{24}$ to obtain the
function
\begin{multline*}
  \frac1{\Delta^4}\Bigl(-21672478
  E_4^4 E_6^2+29822429 E_4 E_6^4+16373825 E_4^7\\
  +  24 E_2 \left(53329 E_4^5 E_6-2096977 E_4^2 E_6^3\right)\\
  +  E_2^2 \left(58616494 E_4^3 E_6^2-23223893 E_4^6-10868825 E_6^4\right)
  \Bigr),
\end{multline*}
which transforms into a Fourier eigenfunction vanishing at $\mathbf{0}$ with
last sign change at $\sqrt8$. This shows $A_+(48)\leq\sqrt8$

The functions
\begin{multline*}
  C_{26}\frac{\phi_{26}}{\Delta^3}=267\frac{E_4^2E_6}{\Delta^2}\llambda+
  \frac{2048}{\theta_{01}^{12} \theta_{10}^{24}\theta_{00}^{24}}\\
 \times \Bigl(570 \theta_{01}^{40}+3705 \theta_{01}^{36} \theta_{10}^4
      +14204 \theta_{01}^{32} \theta_{10}^8
      +41639 \theta_{01}^{28} \theta_{10}^{12}\\
      +71710 \theta_{01}^{24} \theta_{10}^{16}
      +50531 \theta_{01}^{20} \theta_{10}^{20}
      -26351 \theta_{01}^{16}   \theta_{10}^{24}
      -88288 \theta_{01}^{12} \theta_{10}^{28}\\
      -86152 \theta_{01}^8 \theta_{10}^{32}
      -42720 \theta_{01}^4 \theta_{10}^{36}-8544 \theta_{10}^{40}\Bigr)
\end{multline*}
and
\begin{multline*}
  C_{22}\frac{E_4\phi_{22}}{\Delta^3}=3\frac{E_4^2E_6}{\Delta^2}\llambda+
  \frac{1024 \left(\theta_{01}^8+\theta_{01}^4
      \theta_{10}^4+\theta_{10}^8\right)}
  {\theta_{01}^{12} \theta_{10}^{24} \theta_{00}^{24}}\\
    \times\Bigl(30 \theta_{01}^{32}+165 \theta_{01}^{28} \theta_{10}^4
      +395 \theta_{01}^{24} \theta_{10}^8+636 \theta_{01}^{20} \theta_{10}^{12}
      +566 \theta_{01}^{16} \theta_{10}^{16}
      -240 \theta_{01}^{12} \theta_{10}^{20}\\
      -976 \theta_{01}^8 \theta_{10}^{24}
      -768 \theta_{01}^4 \theta_{10}^{28}
      -192 \theta_{10}^{32}\Bigr)  
\end{multline*}
transform into Fourier eigenfunctions with a last sign change at distance
$\sqrt6$. The linear combination
\begin{multline*}
  \frac{\theta_{01}^{16}} {\theta_{10}^{24} \theta_{00}^{24}}
  \Bigl(170 \theta_{01}^{24}
      +1105 \theta_{01}^{20} \theta_{10}^4
      +2678 \theta_{01}^{16} \theta_{10}^8\\
      +2574 \theta_{01}^{12} \theta_{10}^{12}
      -143 \theta_{01}^8 \theta_{10}^{16}-1716
      \theta_{01}^4 \theta_{10}^{20}-572 \theta_{10}^{24}\Bigr)
\end{multline*}
of these two functions transforms into a Fourier eigenfunction vanishing at
$\mathbf{0}$ with sign changes exactly at distances $\sqrt2$, $2$, and
$\sqrt6$. This shows $A_-(48)\leq\sqrt6$.

This should be compared to \cite{Rolen_Wagner2020:note_schwartz_functions}.
There it is erroneously stated that for dimension $48$ only a $+1$
eigenfunction with last sign change at distance $\sqrt6$ can be obtained by
this method. Our result shows that this method produces a $-1$ eigenfunction
with last sign change at $\sqrt6$, whereas the optimal $+1$ eigenfunction
obtained by this method has its last sign change at distance $\sqrt8$.

\appendix
\section{Some preliminaries on modular forms and
  functions}\label{sec:modular}

In this appendix we provide some basic facts about modular and quasimodular
forms, which are required as background for Sections~\ref{sec:even-eigen}
and~\ref{sec:odd}. For further reading and more details we refer to
\cite{Bruinier_Geer_Harder+2008:1_2_3-modular} and
\cite{Cohen_Stroemberg2017:modular_forms,
  Iwaniec1997:topics_classical_automorphic,
  Shimura2012:modular_forms_basics,
  Stein2007:modular_forms_computational,
  Diamond_Shurman2005:first_course_modular}. For introductions to quasimodular
forms we refer to \cite{Zagier2008:elliptic_modular_forms,
  Royer2012:quasimodular_forms_introduction,Choie_Lee2019:jacobi_like_forms}.

We recall the definition of the modular group
\begin{equation*}
  \Gamma=\mathrm{PSL}(2,\mathbb{Z})=\mathrm{SL}(2,\mathbb{Z})/\{\pm I\}
\end{equation*}
and the congruence subgroup
\begin{equation*}
  \Gamma(2)=\{\gamma\in\mathrm{SL}(2,\mathbb{Z})\mid
  \gamma\equiv I\pmod2\}/\{\pm I\},
\end{equation*}
where $I$ denotes the identity matrix.  We also recall that $\Gamma$ is
generated by $T:z\mapsto z+1$ and $S:z\mapsto-\frac1z$; $\Gamma(2)$ is
generated by $ST^2S=\frac{-z}{2z-1}$ and $T^2:z\mapsto z+2$.
\subsection{Modular forms for $\Gamma$, Eisenstein series}
\label{sec:modular-forms-gamma}
A holomorphic function $f:\mathbb{H}\to\mathbb{C}$ is called a \emph{weakly
  holomorphic modular form of weight $k$}, if it satisfies
\begin{equation}
  \label{eq:modular}
  (cz+d)^{-k}f\left(\frac{az+b}{cz+d}\right)=f(z)
\end{equation}
for all $z\in\mathbb{H}$ and all $\bigl( \begin{smallmatrix}a & b\\ c &
  d\end{smallmatrix}\bigr)\in\Gamma$. The space of weakly holomorphic modular
forms is denoted by $\mathcal{M}_k^!(\Gamma)$. This space is non-trivial only
for even values of $k$. A form $f$ is called holomorphic, if
\begin{equation*}
  f(i\infty):=\lim_{\Im z\to+\infty}f(z)
\end{equation*}
exists. The subspace $\mathcal{M}_k(\Gamma)$ of holomorphic modular forms is
non-trivial only for even $k\geq4$. Its dimension equals
\begin{equation*}
  \dim\mathcal{M}_k(\Gamma)=
  \begin{cases}
    \left\lfloor\frac k{12}\right\rfloor&\text{for }k\equiv 2\pmod{12}\\
    \left\lfloor\frac k{12}\right\rfloor+1&\text{otherwise.}
  \end{cases}
\end{equation*}
A holomorphic form $f$ is called a \emph{cusp form}, if $f(i\infty)=0$. The
space of cusp forms is denoted by $\mathcal{S}_k(\Gamma)$.

Prototypical examples of modular forms of weight $2k$ are given by the
Eisenstein series
\begin{equation}
  \label{eq:Eisenstein}
  E_{2k}(z)=\frac1{2\zeta(2k)}\sum_{(m,n)\in\mathbb{Z}\setminus\{(0,0)\}}
  \frac1{(mz+n)^{2k}}=1-\frac{4k}{\mathcal{B}_{2k}}\sum_{n=1}^\infty
  \sigma_{2k-1}(n)q^n
\end{equation}
for $k\geq2$ with
\begin{equation*}
  \sigma_{2k-1}(n)=\sum_{d\mid n}d^{2k-1}
\end{equation*}
and $\mathcal{B}_{2k}$ denoting the Bernoulli numbers. As usual in the context
of modular forms we use the notation $q=e^{2\pi iz}$; $q$ is called the
\emph{nome}.

The two forms
\begin{equation}\label{eq:E4-E6}
     E_4(z)=1+240\sum_{n=1}^\infty \sigma_3(n)q^n,\quad
    E_6(z)=1-504\sum_{n=1}^\infty \sigma_5(n)q^n,
 \end{equation}
of respective weights $4$ and $6$ are especially important, since they generate
the ring of all holomorphic modular forms
\begin{equation*}
  \bigoplus_{k=0}^\infty\mathcal{M}_{2k}(\Gamma)=\mathbb{C}[E_4,E_6].
\end{equation*}

The \emph{modular discriminant} is the prototype of a cusp form
\begin{equation}\label{eq:delta}
  \Delta(z)=\frac1{1728}(E_4(z)^3-E_6(z)^2)=
  q\prod_{n=1}^\infty\left(1-q^n\right)^{24}.
\end{equation}
Its weight is $12$. The relation
\begin{equation*}
  \mathcal{S}_{2k}(\Gamma)=\Delta\mathcal{M}_{2k-12}(\Gamma)
\end{equation*}
characterises the spaces of cusp forms. Furthermore, we have
\begin{equation*}
  \mathcal{M}_{2k}(\Gamma)=\mathbb{C}E_{2k}\oplus\mathcal{S}_{2k}(\Gamma).
\end{equation*}
This decomposition is used in Section~\ref{sec:posit-coeff} to split forms into
an \emph{Eisenstein part} (a form with non-vanishing constant coefficient, for
instance a multiple of $E_{2k}$) and a cusp form.

Klein's modular function
\begin{equation}\label{eq:klein-j}
  j(z)=\frac{E_4(z)^3}{\Delta(z)}
\end{equation}
generates the field of all modular functions (forms of weight $0$)
\begin{equation*}
  \left\{f:\mathbb{H}\to\mathbb{C}\mid \forall\gamma\in\Gamma:
    f(\gamma z)=f(z)\quad\text{and }f\quad\text{meromorphic}\right\}=
  \mathbb{C}(j).
\end{equation*}
This fact is expressed by calling $j$ a \emph{Hauptmodul} for $\Gamma$.

\subsection{Modular forms for $\Gamma(2)$, Theta series}
\label{sec:modular-forms-gamma2}
Powers of the Jacobi theta functions
\begin{align*}
  \theta_{00}(z)&=\sum_{n=-\infty}^\infty e^{\pi in^2z}\\
  \theta_{01}(z)&=\sum_{n=-\infty}^\infty (-1)^ne^{\pi in^2z}\\
  \theta_{10}(z)&=\sum_{n=-\infty}^\infty e^{\pi i(n+\frac12)^2z}
\end{align*}
are examples of modular forms for $\Gamma(2)$. From these definitions it
follows that
\begin{equation}
  \label{eq:theta4}
  \begin{split}
    \theta_{00}(z)^4&=1+\sum_{n=1}^\infty r_4(n)q^{\frac n2}\\
    \theta_{01}(z)^4&=1+\sum_{n=1}^\infty (-1)^nr_4(n)q^{\frac n2},
  \end{split}
\end{equation}
where
\begin{equation} \label{r4def}
r_{4}(k) := |\{\mathbf{x}\in\mathbb{Z}^4\mid\|\mathbf{x}\|^2=k\}|,
\end{equation} 
denotes the number of possibilities to express $n$ as a sum of four squares.
Jacobi's famous four-square theorem gives the following representation of
$r_4(k)$
\begin{equation}\label{r4sigma1}r_{4}(k) =
    8\sigma_{1}(k) - 32\sigma_{1}\left(\frac{k}{4}\right),
\end{equation}
where we use the convention for arithmetic functions that they are defined to
be $0$ for non-integer arguments.

The $\theta$-functions satisfy the transformation formulas
\begin{equation}
\begin{aligned}
  \theta_{00}(Tz)^4&=\theta_{01}(z)^4,&\quad
  z^{-2}\theta_{00}(Sz)^4&=-\theta_{00}(z)\\
  \theta_{01}(Tz)^4&=\theta_{00}(z)^4,&\quad
z^{-2}\theta_{01}(Sz)^4&=-\theta_{10}(z)^4\\
\theta_{10}(Tz)^4&=-\theta_{10}(z)^4,&\quad
z^{-2}\theta_{10}(Sz)^4&=-\theta_{01}(z)^4
\end{aligned}\label{eq:theta}
\end{equation}
and Jacobi's famous relation
\begin{equation}\label{eq:jacobi1}
  \theta_{01}(z)^4+\theta_{10}(z)^4=\theta_{00}(z)^4.
\end{equation}

The modular $\lambda$-function
\begin{equation}\label{eq:lambda-theta}
  \lambda(z)=\frac{\theta_{10}(z)^4}{\theta_{00}(z)^4}=
  \frac{\theta_{10}(z)^4}{\theta_{01}(z)^4+\theta_{10}(z)^4}
\end{equation}
is a Hauptmodul for $\Gamma(2)$ and satisfies
\begin{equation}
  \label{eq:j-lambda}
  j=256\frac{(1-\lambda+\lambda^2)^3}{\lambda^2(1-\lambda)^2}.
\end{equation}
The fact that
\begin{equation*}
  \left[\mathbb{C}(\lambda):\mathbb{C}(j)\right]=6
\end{equation*}
is used in Section~\ref{sec:odd}. The following transformation formulas follow
from \eqref{eq:theta}
\begin{equation}
  \label{eq:lambda-ST}
  \begin{split}
    \lambda(Tz)&=\frac{\lambda(z)}{\lambda(z)-1}\\
    \lambda(Sz)&=1-\lambda(z).
  \end{split}
\end{equation}

The function $\lambda$ is holomorphic on $\mathbb{H}$, attains the
value $1$ at $z=0$, and has no zeros in $\mathbb{H}$.  Hence, we may define
a holomorphic logarithm of $\lambda$ by
\begin{equation} \label{loglambda}
  \llambda(z) : = 2\pi i\int^{z}_{0}\frac{\lambda'(w)}{\lambda(w)}\,dw= \pi
  i\int_0^z\theta_{01}^4(w)\,dw,
  \noindent \end{equation}
where the second equation follows from \eqref{eq:lambda'}. Notice that
$\llambda(it)\in\mathbb{R}_{<0}$ for $t\in\mathbb{R}_{>0}$.
We observe via direct computation with the contour integral and the properties
of $\lambda$ that:
   \begin{align}
\log\lambda(T^{2}z) & = \log\lambda(z) + 2\pi i\\ \label{logform2}
2\log\lambda(Sz) & = \log\lambda(T^{-1}z) - 2\log\lambda(z) + \log\lambda(Tz).
\end{align}
Notice that these equations imply 
\begin{equation}
  \label{eq:log-lambda}
  \log\lambda(z)=\log\lambda(Tz)+\log\lambda(Sz)+\pi i,
\end{equation}
which is used in Section~\ref{sec:odd}.

Using the second equality of \eqref{loglambda} and \eqref{eq:theta4} we obtain
the following expansion of $\log\lambda$ at the cusp $i\infty$:
\begin{equation} \label{logexp}
  \log\lambda(z) = \pi iz + 4\log(2) +
  \sum_{k=1}^{\infty}(-1)^{k}\frac{r_{4}(k)}{k}q^{\frac{k}{2}}, 
\end{equation}
where $r_{4}$ is defined in \eqref{r4def}.  Then \eqref{logform2},
\eqref{logexp} together with \eqref{r4sigma1} give the expansion
\begin{equation}
\label{logexps}
\log\lambda(Sz) =
-16\sum_{k=0}^{\infty}\frac{\sigma_{1}(2k+1)}{2k+1}q^{k+\frac{1}{2}}.
\end{equation}

\subsection{Derivatives of modular
  forms and quasimodular forms}\label{sec:quasimodular-forms}
The ring of modular forms is not closed under differentiation, which can be
seen by differentiating \eqref{eq:modular}. Notice that it is common and
convenient in the context of modular forms to set
$f'=\frac1{2\pi i}\frac{df}{dz}=q\frac{df}{dq}$.  By adjoining the Eisenstein
series of weight $2$
\begin{equation}
  \label{eq:E2}
    E_2(z)=1-24\sum_{n=1}^\infty \sigma_1(n)q^n,
  \end{equation}
  which satisfies the transformation formula
\begin{equation}
  \label{eq:E2S}
  z^{-2}E_2(Sz)=E_2(z)+\frac6{\pi iz},
\end{equation}
we obtain the ring $\mathbb{C}[E_2,E_4,E_6]$ of quasimodular forms, which is
closed under differentiation. This can be seen from
Ramanujan's identities
\begin{equation}
  \label{eq:ramanujan}
  \begin{split}
    E_2'&=\frac1{12}\left(E_2^2-E_4\right)\\
    E_4'&=\frac13\left(E_2E_4-E_6\right)\\
    E_6'&=\frac12\left(E_2E_6-E_4^2\right).
  \end{split}
\end{equation}
Furthermore, we have
\begin{equation*}
  j'=-\frac{E_4^2E_6}\Delta=-\frac{E_6}{E_4}j
\end{equation*}
and
\begin{equation}
  \label{eq:lambda'}
  \lambda'(z)=\frac12\theta_{01}^4(z)\lambda(z),
\end{equation}
as well as
\begin{equation}
  \label{eq:theta'}
  \begin{split}
    \left(\theta_{00}^4\right)'&=\frac16\left(E_2\theta_{00}^4
      -\theta_{01}^8+\theta_{10}^8\right)\\
    \left(\theta_{01}^4\right)'&=\frac16\left(E_2\theta_{01}^4
    -\theta_{01}^8-2\theta_{01}^4\theta_{10}^4\right)\\
    \left(\theta_{10}^4\right)'&=\frac16\left(E_2\theta_{10}^4
      +2\theta_{01}^4\theta_{10}^4+\theta_{10}^8\right).
  \end{split}
\end{equation}

A quasimodular form $f$ of weight $2k$ can be written as
\begin{equation}
  \label{eq:quasi}
  f(z)=\sum_{\ell=0}^k E_2^\ell(z)f_{2k-2\ell}(z),
\end{equation}
where $f_{2k-2\ell}$ is a modular form of weight $2k-2\ell$; the term for
$\ell=k-1$ is of course trivial. Quasimodular forms are invariant under $T$
and transform under $S$ by
\begin{equation*}
  z^{-2k}f(Sz)=\sum_{m=0}^k\left(\frac6{\pi iz}\right)^m
  \sum_{\ell=0}^{k-m}\binom {m+\ell}\ell E_2^{\ell}(z)f_{2k-2\ell-2m}(z).
\end{equation*}
Notice that the terms
\begin{equation*}
  \sum_{\ell=0}^{k-m}\binom {m+\ell}\ell E_2^{\ell}(z)f_{2k-2\ell-2m}(z)
\end{equation*}
are quasimodular forms of weight $2k-2m$. The largest value $\ell$, for which
$f_{2k-2\ell}$ in \eqref{eq:quasi} is non-zero is called the depth of the
quasimodular form.

The derivative of a quasimodular form of weight $2k$ and depth $s$ is a
quasimodular form of weight $2k+2$ and depth at most $s+1$. In particular, the
derivative of a modular form of weight $2k$ is a quasimodular form of weight
$2k+2$ and depth $1$ for $k>0$. 


\begin{ackno}
  This material is based upon work supported by the National Science Foundation
  under Grant No. DMS-1439786 while the authors were in residence at the
  Institute for Computational and Experimental Research in Mathematics in
  Providence, RI, during the
  Spring 2018 semester.\\
  The second author is grateful to Masanobu Kaneko for very helpful email
  conversations and for providing the Master thesis
  \cite{Yamashita2010:construction_extremal_quasimodular}.\\
  The authors are grateful to the anonymous referee for her/his valuable
  suggestions which increased the quality of this paper.
\end{ackno}

\bibliographystyle{amsplain}
\bibliography{hermite}
\end{document}